\documentclass[12pt]{amsart}
\usepackage[all]{xy}
\usepackage{graphicx}
\usepackage{tikz}

\usepackage{amsmath}
\usepackage{amssymb}
\usepackage{amsfonts}
\usepackage{mathrsfs}
\usepackage{mathtools}

\newcommand{\Cdb}{\ensuremath{\mathbb{C}}}

\renewcommand{\H}{\mbox{${\mathcal H}$}}

\renewcommand{\S}{\mbox{${\mathcal S}$}}

\newcommand{\X}{\mbox{${\mathcal X}$}}

\newcommand{\norm}[1]{\Vert#1\Vert}
\newcommand{\bignorm}[1]{\bigl\Vert#1\bigr\Vert}
\newcommand{\Bignorm}[1]{\Bigl\Vert#1\Bigr\Vert}

\newcommand{\Biggnorm}[1]{\Biggl\Vert#1\Biggr\Vert}

\newcommand{\cbnorm}[1]{\Vert#1\Vert_{cb}}

\newtheorem{theorem}{Theorem}[section]
\newtheorem{lemma}[theorem]{Lemma}
\newtheorem{corollary}[theorem]{Corollary}
\newtheorem{proposition}[theorem]{Proposition}

\theoremstyle{remark}
\newtheorem{remark}[theorem]{\bf Remark}
\theoremstyle{definition}

\numberwithin{equation}{section}

\usepackage{color}

\begin{document}

\title[operator multipliers]
{Bilinear operator multipliers into the trace class}

\author[C. Le Merdy]{Christian Le Merdy}
\email{clemerdy@univ-fcomte.fr}
\address{Laboratoire de Math\'ematiques de Besan\c con, UMR 6623,
CNRS, Universit\'e Bourgogne Franche-Comt\'e,
25030 Besan\c{c}on Cedex, France}
\author[I. Todorov]{Ivan G. Todorov}
\email{i.todorov@qub.ac.uk}
\address{Mathematical Sciences Research Center, Queen's University Belfast, Belfast BT7 1NN, United Kingdom}
\author[L. Turowska]{Lyudmila Turowska}
\email{turowska@chalmers.se}
\address{Department of Mathematical Sciences, Chalmers University
of Technology and the University of Gothenburg, Gothenburg SE-412 96, Sweden}

\date{\today}

\maketitle

\begin{abstract}
Given Hilbert spaces $H_1,H_2,H_3$,
we consider bilinear maps defined on the cartesian product $S^2(H_2,H_3)\times S^2(H_1,H_2)$
of spaces of Hilbert-Schmidt operators and valued in either the space
$B(H_1,H_3)$ of bounded operators, or in
the space $S^1(H_1,H_3)$ of trace class operators. We introduce modular properties
of such maps with respect to the commutants of von Neumann algebras $M_i\subset B(H_i)$, $i=1,2,3$,
as well as an appropriate notion of complete boundedness for such maps. We characterize
completely bounded module maps $u\colon S^2(H_2,H_3)\times S^2(H_1,H_2)\to B(H_1,H_3)$
by the membership of a natural symbol of $u$ to the
von Neumann algebra tensor product $M_1\overline{\otimes} M_2^{op}\overline{\otimes}
M_3$. In the case when $M_2$ is injective, we
characterize completely bounded module maps $u\colon S^2(H_2,H_3)\times S^2(H_1,H_2)\to S^1(H_1,H_3)$
by a weak factorization property, which extends to the bilinear setting a famous
description of bimodule linear mappings
going back to Haagerup, Effros-Kishimoto, Smith and Blecher-Smith. We make crucial
use of a theorem of Sinclair-Smith
on completely bounded bilinear
maps valued in an injective von Neumann
algebra, and provide a new proof of it, based on Hilbert $C^*$-modules.
\end{abstract}


\vskip 1cm

\section{Introduction}\label{1Intro} Factorization properties of completely bounded
maps have played a prominent role in the development of operator spaces \cite{BLM, ER, P2}
and in their
applications to Hilbertian operator theory, in particular for the study of
special classes of operators: Schur multipliers, Fourier multipliers on either
commutative or non commutative groups,
module maps, decomposable maps, etc. The main purpose of this paper is to establish new
such factorization properties for some classes of bilinear maps
defined on the cartesian product $S^2(H_2,H_3)\times S^2(H_1,H_2)$
of two spaces of Hilbert-Schmidt operators and valued in their
``product space", namely
the space $S^1(H_1,H_3)$ of trace class operators. This line
of investigation is motivated by the recent characterization of
bounded bilinear Schur multipliers $S^2\times S^2\to
S^1$ proved in \cite{CLPST, CLS}, by various advances on multidimensional
operator multipliers, see \cite{KS, JTT}, and by new developments
on  multiple operator integrals,
see e.g. \cite{AP} and the references therein.

Let $H,K$ be Hilbert spaces and let $M,N$ be von Neumann
algebras acting on $H$ and $K$, respectively.
Let $CB_{(N',M')}(S^1(H,K))$ denote the
Banach space of all
$(N',M')$-bimodule completely bounded maps on
$S^1(H,K)$, equipped with the completely bounded norm $\cbnorm{\, \cdotp}$.
This space is characterized by the following factorization property.

\begin{theorem}\label{Haag}
A bounded map $u\colon S^1(H,K)\to S^1(H,K)$
belongs to $CB_{(N',M')}(S^1(H,K))$ and $\cbnorm{u}\leq 1$ if and only
if there exist an index set $I$, a family $(a_i)_{i\in I}$
of elements of $M$ belonging to the row space
$R_I^w(M)$ and a family $(b_i)_{i\in I}$
of elements of $N$ belonging to the column space
$C_I^w(N)$ such that
$$
u(z)= \sum_{i\in I} b_i z a_i,\qquad
z\in S^1(H,K),
$$
and $\norm{(a_i)_i}_{R_I^w}
\norm{(b_i)_i}_{C_I^w} \leq 1$.
\end{theorem}

We refer to Section \ref{5SS} for the precise definitions of the spaces $R_I^w(M)$
and $C_I^w(N)$. The above theorem is a reformulation of \cite[Theorem 2.2]{BS},
a fundamental factorization result going
back to \cite{EK, H} (see also \cite{Sm}). Indeed let $B(K,H)$
(resp. $S^\infty(K,H)$) denote the space of all bounded operators
(resp. all compact operators) from $K$ into $H$.
Then by standard operator space duality,
the adjoint mapping $u\mapsto u^*$ induces
an isometric isomorphism between
$CB_{(N',M')}(S^1(H,K))$ and the space
$CB_{(M',N')}(S^\infty(K,H), B(K,H))$ of all
$(M',N')$-bimodule completely bounded maps from $S^\infty(K,H)$
into $B(K,H)$. Consequently the description of such maps
provided by \cite[Theorem 2.2]{BS} yields Theorem \ref{Haag}.
Using the so-called
weak$^*$ Haagerup tensor product $\stackrel{w^*h}{\otimes}$
introduced in \cite{BS}, an equivalent formulation of Theorem \ref{Haag}
is that we have a natural isometric $w^*$-homeomorphic identification
\begin{equation}\label{Haag+}
M\stackrel{w^*h}{\otimes} N\,\simeq \,
CB_{(N',M')}(S^1(H,K)).
\end{equation}

In this paper we consider three Hilbert spaces $H_1,H_2,H_3$
as well as von Neumann algebras $M_1,M_2,M_3$ acting on
them. We study
bilinear $(M'_3,M'_2,M'_1)$-module maps
$$
u\colon S^2(H_2,H_3)\times S^2(H_1,H_2)\longrightarrow
S^1(H_1,H_3),
$$
in the sense that $u(Ty,SxR)=Tu(yS,x)R$
for any $x\in S^2(H_1,H_2)$, $y\in S^2(H_2,H_3)$, 
$R\in M'_1$, $S\in M'_2$ and $T\in M'_3$.
In the case when $H_i=L^2(\Omega_i)$ for 
some measure spaces $\Omega_i$, $i=1,2,3$,
and $M_i=L^\infty(\Omega_i)\subset B(L^2(\Omega_i))$ in the usual way, bilinear
$(M'_3,M'_2,M'_1)$-module maps coincide with the bilinear Schur multipliers
discussed in \cite{JTT, CLS}.

On the projective tensor product
$S^2(H_2,H_3)\widehat{\otimes}S^2(H_1,H_2)$,
we introduce a natural operator space structure,
denoted by $\Gamma(H_1,H_2,H_3)$, see  (\ref{2OS-Gamma}).
Our main result, Theorem \ref{6Factorization}, is a characterization,
in the case when $M_2$ is injective, of completely bounded
$(M'_3,M'_2,M'_1)$-module maps $u$ as above
by a weak factorization property, which extends Theorem \ref{Haag}. 
(see Remark \ref{6Recover}).
This characterization is already new
in the non module case (that is, when
$M_i=B(H_i)$ for $i=1,2,3$).

The proof of this result has two steps.
First we establish an isometric and $w^*$-homeomorphic
identification
\begin{equation}\label{1Id}
M_2^{op}\overline{\otimes}\bigl(
M_1\stackrel{w^*h}{\otimes}M_3\bigr)\,\simeq \,
CB_{(M'_3,M'_2,M'_1)}\bigl(\Gamma(H_1,H_2,H_3), S^1(H_1,H_3)\bigr)
\end{equation}
which extends (\ref{Haag+}), see Theorem \ref{3OM-S1}.
Second we make use of a remarkable factorization result
of Sinclair-Smith \cite{SS} on completely bounded bilinear
maps valued in an injective von Neumann
algebra (see Theorem \ref{4SS} for the precise statement), as well as
operator space results, to derive
Theorem \ref{6Factorization} from (\ref{1Id}).

The Sinclair-Smith theorem, which plays a key role in this
paper, was proved in \cite[Theorem 4.4]{SS} using tensor product computations, the
Effros-Lance characterization of semidiscrete von Neumann algebras \cite{EL}
and Connes's fundamental result (completed in \cite{W}) that any injective von Neumann
algebra is semidiscrete.
In Section \ref{5SS} below, we give a new, much shorter proof of Theorem \ref{4SS}
based on Hilbert $C^*$-modules.

The paper also contains a thorough study of completely bounded
bilinear $(M'_3,M'_2,M'_1)$-module maps
$$
u\colon S^2(H_2,H_3)\times S^2(H_1,H_2)\longrightarrow
B(H_1,H_3).
$$
In analogy with (\ref{1Id}) we show that the space of such maps can be identified
with the von Neumann algebra tensor product $M_1\overline{\otimes} M_2^{op}\overline{\otimes}
M_3$, see Corollary \ref{3Mod-B}.

\section{Operator space and duality preliminaries}\label{20S}

We start with some general principles and conventions
which will be used throughout this paper.

Let $E,F$ and $G$ be Banach spaces. We let $E\otimes F$
denote the algebraic tensor product of $E$ and $F$.
We let $B(E,G)$ denote the Banach space
of all bounded operators from $E$ into $G$. We let $B_2(F\times E,G)$
denote the Banach space of all bounded bilinear operators from
$F\times E$ into $G$.

Let $F\widehat{\otimes} E$ be the projective tensor product of $F$ and $E$.
To any $u\in B_2(F\times E,G)$, one can associate a unique
$\widetilde{u}\colon F\otimes E\to G$ satisfying
$$
\widetilde{u}(y\otimes x)=u(y,x),
\qquad x\in E,\ y\in F.
$$
Then $\widetilde{u}$ extends to a bounded operator (still denoted by)
$\widetilde{u}\colon F\widehat{\otimes} E\to G$ and we have equality
$\norm{\widetilde{u}}=\norm{u}$. Then the mapping $u\mapsto \widetilde{u}$
yields an isometric identification
\begin{equation}\label{1Duality1}
B_2(F\times E,G) \,\simeq\, B(F\widehat{\otimes} E,G).
\end{equation}

Consider the case $G=\Cdb$. Then (\ref{1Duality1})
provides an isometric identification
%\begin{equation}\label{1Duality2}
$B_2(F\times E,\Cdb)
\,\simeq\, (F\widehat{\otimes} E)^*$.
%\end{equation}
Now to any bounded bilinear form
$u\colon F\times E\to \Cdb\,$, one can associate two bounded maps
$$
%\begin{equation}\label{1u'}
u'\colon E\longrightarrow F^*
\qquad\hbox{and}\qquad
u''\colon F\longrightarrow E^*
$$
%\end{equation}
defined by $\langle u'(x),y\rangle = u(y,x) = \langle u''(y),x\rangle$
for any $x\in E$ and $y\in F$. Moreover the norms of $u'$ and
$u''$ are equal to the norm of $u$. Hence
%using (\ref{1Duality2}),
the mappings $u\mapsto u'$ and $u\mapsto u''$ yield
isometric identifications
\begin{equation}\label{1Duality3}
(F\widehat{\otimes} E)^*
\,\simeq\, B(E,F^*)\,\simeq\,B(F,E^*).
\end{equation}
We refer to \cite[Chap. 8, Theorem 1 and Corollary 2]{DU} 
for these classical facts.

We assume that the reader is familiar with the basics
of Operator Space Theory and completely bounded maps,
for which we refer to \cite{ER, P2} and \cite[Chap. 1]{BLM}.
However we need to review a few
important definitions and fundamental results
which will be used at length in this paper;
the remainder of this section is devoted to this task.

We will make crucial use of the dual operator
space $E^*$ of an operator space
$E$ as well as of the operator space $CB(E,F)$ of completely bounded
maps from  $E$ into another operator space $F$ (see e.g. \cite[Section 3.2]{ER}).
Whenever $v\colon E\to F$ is a completely bounded map,
its completely bounded norm will be denoted by
$\cbnorm{v}$.

Let $E,F$ be operator spaces. We let $F\stackrel{\frown}{\otimes} E$
denote the operator space projective tensor product of
$F$ and $E$ (here we adopt the notation from \cite[1.5.11]{BLM}).
We will often use the fact that this tensor product is 
commutative
and associative.
The identifications
(\ref{1Duality3}) have operator space analogs. Namely let
$u\colon F\times E\to\Cdb\,$ be a bounded bilinear form. Then
$\widetilde{u}$ extends to a functional on $F\stackrel{\frown}{\otimes} E$
if and only if $u'\colon E\to F^*\,$ is completely bounded, if and
only if $u''\colon F\to E^*$ is completely bounded. In this case
$\cbnorm{u'}=\cbnorm{u''}=\norm{\widetilde{u}}_{(F\stackrel{\frown}{\otimes} E)^*}$.
Thus (\ref{1Duality3}) restricts to isometric identifications
\begin{equation}\label{1Duality4}
(F\stackrel{\frown}{\otimes}E)^*
\,\simeq\, CB(E,F^*)\,\simeq\,CB(F,E^*).
\end{equation}
It turns out that the latter are actually
completely isometric identifications
(see e.g. \cite[Section 7.1]{ER} or \cite[(1.51)]{BLM}).

Let $H,K$ be Hilbert spaces. We let $\overline{K}$ denote the complex conjugate
of $K$.
For any $\xi\in K$, the notation $\overline{\xi}$ stands for
$\xi$ regarded as an element of $\overline{K}$.
We recall the canonical identification $\overline{K}=K^*$. Thus
for any $\xi\in K$ and any $\eta\in H$, $\overline{\xi}\otimes \eta$ may be regarded
as the rank one operator $K\to H$ taking any $\zeta\in K$ to $\langle \zeta,\xi\rangle\eta$.
With this convention, the algebraic tensor product $\overline{K}\otimes H$
is identified with the space of all bounded finite rank operators from $K$
into $H$.

Let $S^1(K,H)$ be the space of trace class operators $v\colon K\to H$, equipped with its
usual norm $\norm{v}_1=tr(\vert v\vert)$. Then $\overline{K}\otimes H$ is a dense subspace
of $S^1(K,H)$ and $\norm{\,\cdotp}_1$ coincides with the Banach space projective norm on
$\overline{K}\otimes H$. Hence we have an isometric identification
\begin{equation}\label{1Proj}
S^1(K,H)\,\simeq\, \overline{K}\widehat{\otimes} H.
\end{equation}

Let
$S^2(K,H)$ be the space of Hilbert-Schmidt operators $v\colon K\to H$, equipped with its
usual Hilbertian norm $\norm{v}_2=\bigl(tr(v^*v)\bigr)^\frac12$.
Then $\overline{K}\otimes H$ is a dense subspace
of $S^2(K,H)$ and $\norm{\,\cdotp}_2$ coincides with the Hilbertian tensor norm on
$\overline{K}\otimes H$. Hence we have an isometric identification
\begin{equation}\label{1HS}
S^2(K,H)\,\simeq\, \overline{K}\stackrel{2}{\otimes} H,
\end{equation}
where the right hand side denotes the Hilbertian tensor product of
$\overline{K}$ and $H$.

Let $S^\infty(H,K)$ denote the space of all compact operators from
$H$ into $K$, equipped with its usual operator space structure. We recall that
through trace duality, we have isometric identifications
\begin{equation}\label{1S11}
S^\infty(H,K)^*
\,\simeq\,
S^1(K,H)
\qquad\hbox{and}\qquad
S^1(K,H)^*
\,\simeq\,
B(H,K).
\end{equation}

Throughout we asssume that $S^1(K,H)$ is equipped with its canonical
operator space structure, so that (\ref{1S11}) holds completely isometrically
(see e.g. \cite[Theorem 3.2.3]{ER}).

Let $E,G$ be Banach spaces and let $j\colon E^*\to G^*$ be a
$w^*$-continuous isometry. Then its range $j(E^*)$ is
$w^*$-closed, hence  $j(E^*)$ is a dual space.
Further $j$ induces a $w^*$-$w^*$-homeomorphism
between $E^*$ and $j(E^*)$ (see e.g. \cite[A.2.5]{BLM}).
Thus $j$ allows to identify  $E^*$ and $j(E^*)$ as dual
Banach spaces.
In this case, we will say that $j$ induces a $w^*$-continuous
isometric identification between $E^*$ and $j(E^*)$. If
$E,G$ are operator spaces and $j$ is a complete isometry,
then $j(E^*)$ is a dual operator space and we will call $j$ a
$w^*$-continuous completely
isometric identification between $E^*$ and $j(E^*)$.

Let $E,F$ be operator spaces
and consider $w^*$-continuous completely isometric embeddings
\begin{equation}\label{1Rep}
E^*\subset B(H)
\qquad\hbox{and}\qquad
F^*\subset B(K),
\end{equation}
for some Hilbert spaces $H,K$ (see e.g. \cite[Prop. 3.2.4]{ER}).
The normal spatial tensor product of the dual operator
spaces $F^*$ and $E^*$ is defined as
the $w^*$-closure of $F^*\otimes E^*$ into the von Neumann algebra
$B(K)\overline{\otimes} B(H)$ and is denoted by
$$
F^*\overline{\otimes} E^*.
$$
This is a dual operator space.
It turns out that its definition
does not depend on the specific embeddings
(\ref{1Rep}), see e.g. \cite[p. 135]{ER}.

We note for further use that the natural embedding
$B(K)\otimes B(H)\subset B(K\stackrel{2}{\otimes} H)$
extends to a $w^*$-continuous
completely isometric identification
\begin{equation}\label{1Normal}
B(K)\overline{\otimes} B(H)
\,\simeq\,
B(K\stackrel{2}{\otimes} H).
\end{equation}

To deal with normal spatial tensor products, it is convenient to
use the so-called slice maps.
Take any $\lambda\in  S^1(K)$ and consider it as
a $w^*$-continuous functional $\lambda\colon B(K)\to\Cdb\,$.
Then the operator $\lambda\otimes I_{B(H)}$
extends to a (necessarily unique) $w^*$-continuous bounded map
$$
\ell_\lambda\colon B(K)\overline{\otimes} B(H)\longrightarrow B(H).
$$
Likewise, any $\mu\in S^1(H)$ can be considered as
a $w^*$-continuous functional $\mu\colon B(H)\to\Cdb\,$ and
$I_{B(K)}\otimes \mu$
extends to a $w^*$-continuous bounded map
$$
r_\mu\colon B(K)\overline{\otimes} B(H)\longrightarrow B(K).
$$
Then we have the following properties (for which we refer to either
\cite[Lemma 7.2.2]{ER} and its proof, or \cite[1.5.2]{BLM}).

\begin{lemma}\label{1Slice}
Let $z\in B(K)\overline{\otimes} B(H)$. The linear mappings
$$
z'\colon S^1(H)\longrightarrow B(K)
\qquad\hbox{and}\qquad
z''\colon S^1(K)\longrightarrow B(H)
$$
defined by $z'(\mu)=r_\mu(z)$ and
$z''(\lambda)=\ell_\lambda(z)$
are completely bounded.

Further the mappings $z\mapsto z'$ and $z\mapsto z''$ are
$w^*$-continuous completely isometric isomorphisms
from $B(K)\overline{\otimes} B(H)$ onto $CB(S^1(H),B(K))$
and $CB(S^1(K),B(H))$, respectively.
\end{lemma}

According to (\ref{1Duality4}), an equivalent formulation of the above lemma is that
\begin{equation}\label{1S1S1}
\bigl(S^1(K)\stackrel{\frown}{\otimes} S^1(H)\bigr)^*\,\simeq\,
B(K)\overline{\otimes} B(H)
\end{equation}
$w^*$-continuously and completely isometrically.

Recall (\ref{1Rep}).
The space of all $z\in B(K)\overline{\otimes} B(H)$ such that
$z'$ is valued in $F^*$ and $z''$ is valued in $E^*$ is usually called
the normal Fubini tensor product of $F^*$ and $E^*$. This subspace is $w^*$-continuously completely isometric
to $CB(E,F^*)$ (equivalently to $CB(F,E^*)$, by (\ref{1Duality4})).
Indeed we may regard $CB(E,F^*)$ as the subspace of $CB(S^1(H),B(K))$ of
all $w\colon S^1(H)\to B(K)$ such that $w$ is valued in $F^*$ and $w$ vanishes on $E^{*}_{\perp}$.
Then it is not hard to see that $z$ belongs to the normal Fubini tensor product of $F^*$ and $E^*$
if and only if $z'$ belongs to $CB(E,F^*)$.
We refer to \cite[Theorem 7.2.3]{ER} for these facts.

It is easy to check
that the normal Fubini tensor product of $F^*$ and $E^*$
contains $F^*\overline{\otimes} E^*$. This yields a $w^*$-continuous
completely isometric embedding
$$
F^*\overline{\otimes} E^*\,\subset\, CB(E,F^*).
$$
However this inclusion may be strict. The next lemma provides a list of cases
when the inclusion is an equality.  We refer the reader to \cite[Sections 7.2 and 11.2]{ER}
for the proofs.

Whenever $M$ is a von Neumann algebra, we let
$M_*$ denote its (unique) predual. We equip it with its
natural operator space structure, so that $M=(M_*)^*$
completely isometrically (see e.g. \cite[Section 2.5]{P2}
or \cite[Lemma 1.4.6]{BLM}).

\begin{lemma}\label{1Injective1}
\

\begin{itemize}
\item [(a)] For any von Neumann algebras $M,N$, we have
$$
N\overline{\otimes}M\,\simeq\, CB(M_*,N).
$$
\item [(b)] For any injective von Neumann algebra $M$ and for any
operator space $E$, we have
$$
M\overline{\otimes} E^*\,\simeq\, CB(E,M).
$$
\item [(c)] For any Hilbert spaces $H,K$ and for any
operator space $E$, we have
$$
B(H,K)\overline{\otimes} E^*\,\simeq\, CB(E,B(H,K)).
$$
\end{itemize}
\end{lemma}

Let $K$ be a Hilbert space. We let $\{K\}_c$ (resp. $\{K\}_r$)
denote the column operator space (resp. the row operator space) over $K$.
We recall that through the canonical identification $K^*=\overline{K}$, we have
$$
\{K\}_c^* = \{\overline{K}\}_r
\quad\hbox{and}\quad
\{K\}_r^* = \{\overline{K}\}_c
\qquad\hbox{completely isometrically}.
$$
(See e.g. \cite[Section 3.4]{ER}.)

We let $F\stackrel{h}{\otimes}E$ denote
the Haagerup tensor product of a couple $(F,E)$
of operator spaces. We will use the
fact that this is an associative tensor product.

Let
$\theta\colon F\times E\to\Cdb\,$ be a bounded
bilinear form. Then $\theta$
extends to an element
of $(F\stackrel{h}{\otimes}E)^*$  if and only if there
exist a Hilbert space $\H$ and two completely bounded
maps $\alpha\colon E\to \{\H\}_c\,$ and $\beta\colon F\to
\{\overline{\H}\}_r\,$
such that $\theta(y,x)= \langle \alpha(x),\beta(y)\rangle$
for any $x\in E$ and
any $y\in F$ (see e.g. \cite[Corollary 9.4.2]{ER}).

The Haagerup tensor product is projective. This means that if
$p\colon E\to E_1$ and $q\colon F\to F_1$ are complete quotient maps, then
$q\otimes p$ extends to a (necessarily unique) complete quotient map
$F\stackrel{h}{\otimes}E\to
F_1\stackrel{h}{\otimes}E_1$. Taking the adjoint of the latter, we
obtain a $w^*$-continuous
completely isometric embedding
\begin{equation}\label{1Embed}
(F_1\stackrel{h}{\otimes}E_1)^*
\,\subset\,
(F\stackrel{h}{\otimes}E)^*.
\end{equation}

\begin{lemma}\label{1Slice-H}
Let $E,F,E_1,F_1$ be operator spaces as above and let
$\theta\in (F\stackrel{h}{\otimes}E)^*$. Let
$$
\theta'\colon E\longrightarrow F^*
\qquad\hbox{and}\qquad
\theta''\colon F\longrightarrow E^*
$$
be the bounded linear maps associated to $\theta$.
Then $\theta\in (F_1\stackrel{h}{\otimes}E_1)^*$ (in the sense given by (\ref{1Embed}))
if and only if $\theta'$ is valued in $F_1^*$ and $\theta''$ is valued in $E_1^*$.
\end{lemma}

\begin{proof}
If $\theta\in (F_1\stackrel{h}{\otimes}E_1)^*$, then
$\theta(y,x)=0$ if either $x\in{\rm Ker}(p)$ or
$y\in{\rm Ker}(q)$. Hence $\langle \theta''(y),x\rangle=0$
for any $(y,x)\in F\times {\rm Ker}(p)$ and
$\langle \theta'(x),y\rangle=0$ for any $(y,x)\in {\rm Ker}(q)\times E$.
Hence $\theta''$ is valued in $E_1^*={\rm Ker}(p)^\perp$
and $\theta'$ is valued in $F_1^*={\rm Ker}(q)^\perp$.

Assume conversely that $\theta'$ is valued in $F_1^*$ and that
$\theta''$ is valued in $E_1^*$.
Let $\alpha\colon E\to \{\H\}_c\,$ and  $\beta\colon F\to
\{\overline{\H}\}_r\,$
be completely bounded maps, for some Hilbert space $\H$,
such that $\theta(y,x)= \langle \alpha(x),\beta(y)\rangle$
for any $x\in E$ and
any $y\in F$. Changing $\H$ into the closure of the range of $\alpha$, we may assume that
$\alpha$ has dense range. Next changing $\H$ into the closure of the
(conjugate of) the range of $\beta$, we may
actually assume that both $\alpha$ and $\beta$ have dense range.

The assumption on $\theta'$
means that $\langle \alpha(x),\beta(y)\rangle=0$
for any $x\in E$ and any $y\in {\rm Ker}(q)$. Since $\alpha$ has dense range
this means that $\beta$ vanishes on ${\rm Ker}(q)$. Likewise
the assumption on $\theta''$
means that $\alpha$ vanishes on ${\rm Ker}(p)$.
We may therefore consider $\alpha_1\colon E_1\to \{\H\}_c$ and
$\beta_1\colon F_1\to \{\overline{\H}\}_r$
induced by $\alpha$ and $\beta$, that is,
$\alpha=\alpha_1\circ p$ and $\beta=\beta_1\circ q$.
Further $\alpha_1$ and $\beta_1$ are completely
bounded, hence the bilinear mapping
$(y_1,x_1)\mapsto \langle \alpha_1(x_1),\beta_1(y_1)\rangle$
is an element of $(F_1\stackrel{h}{\otimes}E_1)^*$. By construction
it identifies with $\theta$ in the embedding (\ref{1Embed}),
hence $\theta$ belongs to $(F_1\stackrel{h}{\otimes}E_1)^*$.
\end{proof}

We will need the so-called weak$^*$ Haagerup tensor product
of two dual operator spaces \cite{BS}. It can be defined by
\begin{equation}\label{1w*h}
F^*\stackrel{w^*h}{\otimes}E^* \,=\, (F\stackrel{h}{\otimes}E)^*.
\end{equation}
The reason why this dual space can be considered as a tensor product over the
couple $(F^*,E^*)$ is discussed in 
\cite[1.6.9]{BLM}.

We now recall a few tensor product identities involving the
operator space projective tensor product and the Haagerup
tensor product.

\begin{proposition}\label{1Recap} Let $E$ be an operator space
and let $H,K$ be two Hilbert spaces.
\begin{itemize}
\item [(a)] We have completely isometric identifications
%\begin{equation}\label{1Proj=Haag}
$$
\{K\}_r\stackrel{\frown}{\otimes} E
\,\simeq\,
\{K\}_r\stackrel{h}{\otimes} E
\qquad\hbox{and}\qquad
E\stackrel{\frown}{\otimes} \{H\}_c
\,\simeq\,
E\stackrel{h}{\otimes} \{H\}_c.
$$
%\end{equation}

\item [(b)] We have completely isometric identifications
%\begin{equation}\label{1Hilbert}
$$
\{K\}_r\stackrel{\frown}{\otimes} \{H\}_r
\,\simeq\,
\{K\stackrel{2}{\otimes} H\}_r
\qquad\hbox{and}\qquad
\{K\}_c\stackrel{\frown}{\otimes} \{H\}_c
\,\simeq\,
\{K\stackrel{2}{\otimes} H\}_c.
$$
%\end{equation}

\item [(c)] The embedding $\overline{K}\otimes H\subset S^1(K,H)$
extends to completely isometric identifications
\begin{equation}\label{1RC}
S^1(K,H)
\,\simeq\,
\{\overline{K}\}_r\stackrel{\frown}{\otimes}\{H\}_c
\end{equation}
and
\begin{equation}\label{1REC}
S^1(K,H)\stackrel{\frown}{\otimes} E
\,\simeq\,
\{\overline{K}\}_r\stackrel{\frown}{\otimes} E
\stackrel{\frown}{\otimes}\{H\}_c.
\end{equation}

\item [(d)] To any $u\colon E\to B(H,K)$, associate
$\theta_u\colon \overline{K}\otimes E\otimes H\to\Cdb\,$ by
letting $\theta_u(\overline{\xi}\otimes x\otimes \eta )
=\langle u(x)\eta,\xi\rangle$, for any $x\in E,\eta\in H,\xi\in K$.
Then $u\mapsto\theta_u$ extends to a $w^*$-continuous completely
isometric identification
%\begin{equation}\label{1CB(E,B)}
$$
\bigl(\{\overline{K}\}_r\stackrel{\frown}{\otimes} E
\stackrel{\frown}{\otimes}\{H\}_c\bigr)^*
\,\simeq\,
CB(E,B(H,K)).
$$
%\end{equation}
\end{itemize}
\end{proposition}

\begin{proof}
We refer to \cite[Proposition 9.3.2]{ER} for (a)
and to \cite[Proposition 9.3.5]{ER} for (b).
Formula (\ref{1RC}) follows from \cite[Proposition 9.3.4]{ER}
and (a), and formula (\ref{1REC}) follows by the comutativity
of the operator space projective tensor product.
Finally (d) is a consequence of (\ref{1REC}), (\ref{1S11}) and (\ref{1Duality4}).
\end{proof}

\begin{remark}\label{1Equal}
Comparing (\ref{1RC})
with (\ref{1Proj}), we note
that at the Banach space level,
the operator space projective tensor product of a row and a column
Hilbert space coincides with their Banach space projective tensor product.
\end{remark}

\begin{remark}\label{1RankOne} For any $\eta\in H$ and $\xi\in K$, let
$T_{\eta,\xi}\in B(H,K)$ be the rank one operator defined by
$$
T_{\eta,\xi}(\zeta)=\langle \zeta,\eta\rangle\, \xi,\qquad\zeta\in H.
$$
When we consider this operator as an element of $S^\infty(H,K)$
or $B(H,K)$, it is convenient to identify it with $\xi\otimes\overline{\eta}
\in K\otimes\overline{H}$, and hence to regard $K\otimes\overline{H}$
as a subspace of $S^\infty(H,K)$. This convention is different from the
one used so far when we had to represent rank one (more generally, finite
rank) operators as elements of the trace class or of the Hilbert-Schmidt class.

The rationale for this is that the
trace duality providing (\ref{1S11}) extends the natural duality
between $K\otimes\overline{H}$ and $\overline{K}\otimes H$.
Then the embedding $K\otimes\overline{H}\subset
S^\infty(H,K)$ extends to a completely isometric identification
\begin{equation}\label{1Compact}
S^\infty (H,K)\,\simeq\, \{K\}_c\stackrel{h}{\otimes} \{\overline{H}\}_r.
\end{equation}
(See e.g. \cite[Proposition 9.3.4]{ER}.)
\end{remark}

If $A$ is any $C^*$-algebra, the so-called opposite
$C^*$-algebra $A^{op}$ is the involutive Banach space
$A$ equipped with its reversed
multiplication $(a,b)\mapsto ba$. Note that as an operator
space, $A^{op}$ is not (in general) the same as $A$,
that is, the identity mapping $A\to A^{op}$ is
not a complete isometry. See e.g. \cite[Theorem 2.2]{Roy}
for more about this.

In the case when $A=B(H)$, we have the following
well-known description (see e.g. \cite[Sections 2.9 and 2.10]{P2}).

\begin{lemma}\label{1Opp}
Let $H$ be a Hilbert space. For any $S\in B(H)$, define
$$
\widehat{S}(\overline{h}) = \overline{S^*(h)},\qquad
h\in H.
$$
Then $S\mapsto \widehat{S}$ is a $*$-isomorphism from $B(H)^{op}$
onto $B(\overline{H})$.
\end{lemma}

In the sequel we will use the operator space $M_{*}^{op}$ for any von Neumann algebra $M$.
This is both the predual operator space of $M^{op}$ and the opposite operator space of
$M_*$, in the sense of \cite[Section 2.10]{P2}.

\section{Operator multipliers into the trace class}\label{3OM}
Let $H_1,H_2,H_3$ be three Hilbert spaces.
Using (\ref{1HS}), we let
$$
\Theta\colon H_1\stackrel{2}{\otimes}\overline{H_2}\stackrel{2}{\otimes} H_3\longrightarrow
S^2\bigl(S^2(H_1,H_2),H_3\bigr)
$$
be the unitary operator obtained by first identifying $H_1\stackrel{2}{\otimes}\overline{H_2}$
with $\overline{S^2(H_1,H_2)}$, and then identifying $\overline{S^2(H_1,H_2)}\stackrel{2}{\otimes}
H_3$ with $S^2\bigl(S^2(H_1,H_2),H_3\bigr)$.

For any $\varphi\in B\bigl(H_1\stackrel{2}{\otimes}\overline{H_2}
\stackrel{2}{\otimes} H_3\bigr)$, one may define
a bounded bilinear map
$$
\tau_\varphi\colon  S^2(H_2,H_3)\times S^2(H_1,H_2)\longrightarrow B(H_1,H_3)
$$
by
$$
\bigl[\tau_\varphi(y,x)\bigr](h) = \Theta\bigl[\varphi(h\otimes y)\bigr](x),
\qquad x\in S^2(H_1,H_2), \ y\in S^2(H_2,H_3),\ h\in H_1.
$$
On the right hand side of the above
equality, $y$ is regarded as an element of
$\overline{H_2}\stackrel{2}{\otimes}H_3$, and hence $h\otimes y$
is an element of $H_1\stackrel{2}{\otimes}\overline{H_2}
\stackrel{2}{\otimes} H_3$.
It is clear that
$$
\bignorm{\bigl[\tau_\varphi(y,x)\bigr](h)}
\leq \norm{\varphi}\norm{x}_2\norm{y}_2\norm{h}.
$$
Consequently, the above construction defines a contraction
\begin{equation}\label{2Sigma1}
\tau\colon
B\bigl(H_1\stackrel{2}{\otimes}\overline{H_2}\stackrel{2}{\otimes} H_3\bigr)
\longrightarrow
B_2\bigl(S^2(H_2,H_3)\times S^2(H_1,H_2), B(H_1,H_3)\bigr).
\end{equation}
The bilinear maps $\tau_\varphi$ were introduced in \cite{JTT} 
(however the latter paper focuses on the case when 
$\bignorm{\bigl[\tau_\varphi(y,x)\bigr](h)}
\leq D\norm{x}\norm{y}\norm{h}$ for some constant $D>0$).
We call $\tau_\varphi$ an operator multiplier and we say that
$\varphi$ is the symbol of $\tau_\varphi$. We refer to \cite{JTT} for
$m$-linear versions of such operators for arbitrary $m\geq 2$.

We note that by (\ref{1Normal}) and Lemma \ref{1Opp},
we have a von Neumann algebra identification
\begin{equation}\label{2VN}
B\bigl(H_1\stackrel{2}{\otimes}\overline{H_2}\stackrel{2}{\otimes} H_3\bigr)
\,\simeq\,
B(H_1)\overline{\otimes} B(H_2)^{op}\overline{\otimes} B(H_3).
\end{equation}
In the sequel we will make no difference between these two von Neumann algebras.
In particular, we will consider symbols $\varphi$ of operator multipliers as elements
of $B(H_1)\overline{\otimes} B(H_2)^{op}\overline{\otimes} B(H_3)$.

One can check (see \cite{JTT}) that for any $R\in B(H_1)$, $S\in B(H_2)$ and
$T\in B(H_3)$, we have
\begin{equation}\label{2Sigma2}
\tau_{R\otimes S\otimes T}(y,x) = TySxR,
\qquad x\in S^2(H_1,H_2), \ y\in S^2(H_2,H_3).
\end{equation}
Note that in this identity,
$S$ is regarded as an element of $B(H_2)^{op}$
at the left-hand side 
and as an element of $B(H_2)$ at the right-hand side.

We now define the operator space
\begin{equation}\label{2OS-Gamma}
\Gamma(H_1,H_2,H_3)\, =\, \{S^2(H_2,H_3)\}_c
\stackrel{\frown}{\otimes} \{S^2(H_1,H_2)\}_r.
\end{equation}
According to Remark \ref{1Equal}, $\Gamma(H_1,H_2,H_3)$
coincides, at the Banach
space level, with the projective tensor
product of $S^2(H_2,H_3)$ and $S^2(H_1,H_2)$. Hence
\begin{equation}\label{2Equal}
B_2\bigl(S^2(H_2,H_3)\times S^2(H_1,H_2), B(H_1,H_3)\bigr)
\,\simeq\, B\bigl(\Gamma(H_1,H_2,H_3), B(H_1,H_3)\bigr)
\end{equation}
by (\ref{1Duality1}). In the sequel for any
$u\colon S^2(H_2,H_3)\times S^2(H_1,H_2)\to B(H_1,H_3)$,
we let
$$
\widetilde{u}\colon
\Gamma(H_1,H_2,H_3)\longrightarrow B(H_1,H_3)
$$
denote its associated linear map.

The next proposition shows that under the identification (\ref{2Equal}),
the range of $\tau$
coincides with the space of completely bounded maps from
$\Gamma(H_1,H_2,H_3)$ into $B(H_1,H_3)$.

\begin{proposition}\label{2OM-B}
Let $u\colon S^2(H_2,H_3)\times S^2(H_1,H_2)\to B(H_1,H_3)$
be a bounded bilinear map. Then $\widetilde{u}\colon
\Gamma(H_1,H_2,H_3)\to B(H_1,H_3)$ is completely bounded if and only if
there exists $\varphi$ in
$B(H_1)\overline{\otimes} B(H_2)^{op}\overline{\otimes} B(H_3)$
such that $u=\tau_\varphi$. Further
$\tau$ provides a $w^*$-continuous
completely isometric identification
\begin{equation}\label{2OM-B1}
B(H_1)\overline{\otimes} B(H_2)^{op}\overline{\otimes} B(H_3)
\,\simeq\, CB\bigl(\Gamma(H_1,H_2,H_3), B(H_1,H_3)\bigr).
\end{equation}
\end{proposition}

\begin{proof}
For convenience we set
$$
\H=H_1\stackrel{2}{\otimes}\overline{H_2}
\stackrel{2}{\otimes} H_3.
$$
By (\ref{1HS}) and Proposition \ref{1Recap} (b),
we have
$$
\{S^2(H_1,H_2)\}_r\,\simeq\,\{\overline{H_1}\}_r
\stackrel{\frown}{\otimes}\{H_2\}_r
\qquad\hbox{and}\qquad
\{S^2(H_2,H_3)\}_c\,\simeq\,\{\overline{H_2}\}_c
\stackrel{\frown}{\otimes}\{H_3\}_c
$$
completely isometrically.
Hence applying (\ref{2OS-Gamma}), we have
\begin{equation}\label{2Gamma}
\Gamma(H_1,H_2,H_3) \,\simeq\,
\{\overline{H_2}\}_c
\stackrel{\frown}{\otimes}\{H_3\}_c
\stackrel{\frown}{\otimes}
\{\overline{H_1}\}_r
\stackrel{\frown}{\otimes}\{H_2\}_r
\end{equation}
completely isometrically.
Using the commutativity of the operator space projective
tensor product, we deduce a
completely isometric identification
$$
\{\overline{H_3}\}_r\stackrel{\frown}{\otimes}
\Gamma(H_1,H_2,H_3)\stackrel{\frown}{\otimes}
\{H_1\}_c
\,\simeq\, \{\overline{H_1}\}_r\stackrel{\frown}{\otimes}
\{H_2\}_r\stackrel{\frown}{\otimes}\{\overline{H_3}\}_r
\stackrel{\frown}{\otimes}
\{H_1\}_c
\stackrel{\frown}{\otimes}\{\overline{H_2}\}_c\stackrel{\frown}{\otimes}\{H_3\}_c.
$$
Using Proposition \ref{1Recap} (b) again, we have
$$
\{H_1\}_c
\stackrel{\frown}{\otimes}\{\overline{H_2}\}_c\stackrel{\frown}{\otimes}\{H_3\}_c
\simeq\{\H\}_c
\qquad\hbox{and}\qquad \{\overline{H_1}\}_r\stackrel{\frown}{\otimes}
\{H_2\}_r\stackrel{\frown}{\otimes}\{\overline{H_3}\}_r \simeq\{\overline{\H}\}_r
$$
completely isometrically.
By Proposition \ref{1Recap}  (c),
this yields a completely isometric identification
\begin{equation}\label{2Ident1}
\{\overline{H_3}\}_r\stackrel{\frown}{\otimes}
\Gamma(H_1,H_2,H_3)\stackrel{\frown}{\otimes}
\{H_1\}_c
\simeq S^1\bigl(\H).
\end{equation}
Passing to the duals, using (\ref{1S11}) and Proposition \ref{1Recap} (d),
we deduce a $w^*$-continuous
completely isometric identification
$$
B(\H)\,\simeq\, CB\bigl(\Gamma(H_1,H_2,H_3), B(H_1,H_3)\bigr).
$$
Combining with (\ref{2VN}), we deduce a $w^*$-continuous,
completely isometric onto
mapping
\begin{equation}\label{2J}
J\colon B(H_1)\overline{\otimes} B(H_2)^{op}\overline{\otimes} B(H_3)
\longrightarrow CB\bigl(\Gamma(H_1,H_2,H_3), B(H_1,H_3)\bigr).
\end{equation}

Now to establish the proposition it suffices to check that
\begin{equation}\label{2J=Tau}
J(\varphi)=\widetilde{\tau}_\varphi
\end{equation}
for any
$\varphi\in B(H_1)\overline{\otimes} B(H_2)^{op}\overline{\otimes} B(H_3)$.

We claim that it suffices to prove (\ref{2J=Tau}) in the
case when $\varphi$ belongs to the algebraic tensor
product $B(H_1)\otimes B(H_2)^{op}\otimes B(H_3)$.
Indeed let $\varphi\in B(H_1)\overline{\otimes} B(H_2)^{op}\overline{\otimes} B(H_3)$,
let $x\in S_2(H_1,H_2)$, $y\in S_2(H_2,H_3)$ and
$h\in H_1$. Assume that $(\varphi_t)_t$ is a net of
$B(H_1)\otimes B(H_2)^{op}\otimes B(H_3)$ converging
to $\varphi$ in the $w^*$-topology. Then
$\varphi_t(h\otimes y)\to \varphi(h\otimes y)$ in the weak topology
of $H_1\stackrel{2}{\otimes} \overline{H_2}\stackrel{2}{\otimes} H_3$. Hence
$\Theta[\varphi_t(h\otimes y)]\to \Theta[\varphi(h\otimes y)]$ in the weak topology
of $S^2(S^2(H_1,H_2),H_3)$, which implies that
$\Theta[\varphi_t(h\otimes y)](x)
\to \Theta[\varphi(h\otimes y)](x)$ in the weak topology
of $H_3$. Equivalently, $[\tau_{\varphi_t}(y,x)](h)\to [\tau_{\varphi}(y,x)](h)$
weakly. Since $J$ is $w^*$-continuous, we also have, by similar arguments, that
$[J(\varphi_t)(y\otimes x)](h)\to [J(\varphi)(y\otimes x)](h)$ weakly. Hence
if $J(\varphi_t)=\widetilde{\tau}_{\varphi_t}$ for any $t$, we have
$J(\varphi)=\widetilde{\tau}_\varphi$ as well.

Moreover by linearity, it suffices to prove (\ref{2J=Tau})
when $\varphi= R\otimes S \otimes T$ for some
$R\in B(H_1)$, $S\in B(H_2)$ and
$T\in B(H_3)$.
In view of (\ref{2Sigma2}), it therefore suffices to show that
\begin{equation}\label{2Check}
J(R\otimes S\otimes T)(y\otimes x ) = TySxR,
\end{equation}
for any $R\in B(H_1)$, $S\in B(H_2)$,
$T\in B(H_3)$,
$x\in S^2(H_1,H_2)$ and $y\in S^2(H_2,H_3)$.
Since $J$ is linear and $w^*$-continuous it actually suffices to
prove (\ref{2Check}) when $R$, $S$, $T$, $x$ and $y$ are rank one.

For $i=1,2,3$, let $\xi_i,\eta_i,h_i,k_i\in H_i$ and consider
$x=\overline{\xi_1}\otimes\eta_2$ and
$y= \overline{\xi_2}\otimes\eta_3$, as well as the operators
$R=h_1\otimes\overline{k_1}$, $S=h_2\otimes\overline{k_2}$ and
$T=h_3\otimes\overline{k_3}$ (see Remark \ref{1RankOne} for the
use of these tensor product notations).
Then let
$$
\alpha = \overline{\xi_1}\otimes\eta_2\otimes\overline{\xi_3}
\otimes \eta_1\otimes\overline{\xi_2}\otimes \eta_3\, \in
\, (\overline{H_1}\otimes
H_2\otimes \overline{H_3})\otimes (H_1\otimes \overline{H_2}\otimes H_3)
\,\subset\,
S^1(\H)
$$
and let
$$
\beta= h_1\otimes \overline{k_2}\otimes
h_3\otimes \overline{k_1}\otimes h_2\otimes \overline{k_3}
\, \in
\, (H_1\otimes \overline{H_2}\otimes H_3)\otimes  (\overline{H_1}\otimes
H_2\otimes \overline{H_3})\,\subset\,
B(\H).
$$
In the identification (\ref{2Ident1}), $\overline{\xi_3}\otimes y\otimes x
\otimes \eta_1$ corresponds to $\alpha$ whereas in the identification (\ref{2VN}),
$R\otimes S\otimes T$ corresponds to $\beta$. Hence
\begin{align*}
\bigl\langle\bigl[J(R\otimes S\otimes T)(y\otimes x)](\eta_1),\xi_3\bigr\rangle
& ={\rm tr}(\alpha\beta)\\
& =\langle h_1,\xi_1\rangle\langle \eta_2,k_2\rangle
\langle h_3,\xi_3\rangle\langle \eta_1,k_1\rangle
\langle h_2,\xi_2\rangle\langle \eta_3,k_3\rangle.
\end{align*}
On the other hand,
$$
TySxR = \langle h_1,\xi_1\rangle\langle \eta_2,k_2\rangle
\langle h_2,\xi_2\rangle\langle \eta_3,k_3\rangle
h_3\otimes\overline{k}_1,
$$
hence
\begin{equation}\label{2Trace}
\langle TySxR(\eta_1),\xi_3\rangle = \langle h_1,\xi_1\rangle\langle \eta_2,k_2\rangle
\langle h_3,\xi_3\rangle\langle \eta_1,k_1\rangle
\langle h_2,\xi_2\rangle\langle \eta_3,k_3\rangle.
\end{equation}
This proves the desired equality.
\end{proof}

\begin{remark}\label{2Rk} Using (\ref{1S1S1}) twice we have a
$w^*$-continuous completely
isometric identification
\begin{equation}\label{2Rk1}
B(H_1)\overline{\otimes} B(H_2)^{op}\overline{\otimes} B(H_3)\,\simeq\,
\bigl(S^1(H_1)\stackrel{\frown}{\otimes}
S^1(H_2)^{op}\stackrel{\frown}{\otimes}
S^1(H_3)\bigr)^*.
\end{equation}

Let $\varphi\in B(H_1)\overline{\otimes} B(H_2)^{op}\overline{\otimes} B(H_3)$
and let $u=\tau_\varphi$. Let
$\xi_1,\eta_1\in H_1$, $\xi_2,\eta_2\in H_2$ and $\xi_3,\eta_3\in H_3$ and
regard $\overline{\xi_i}\otimes\eta_i$ as an element of $S^1(H_i)$ for $i=1,2,3$.
According to (\ref{2Rk1}) we may consider the action of $\varphi$
on $\overline{\xi_1}\otimes\eta_1\otimes\eta_2\otimes
\overline{\xi_2}\otimes\overline{\xi_3}\otimes\eta_3$.
Then we have
$$
\langle\varphi, \overline{\xi_1}\otimes\eta_1\otimes\eta_2\otimes
\overline{\xi_2}\otimes\overline{\xi_3}\otimes\eta_3\rangle\,=\,
\bigl\langle \bigl[u(\overline{\xi_2}\otimes\eta_3, \overline{\xi_1}\otimes\eta_2)\bigr]
(\eta_1),\xi_3\bigr\rangle.
$$
Indeed this follows from the arguments
in the proof of Proposition \ref{2OM-B}. Details are left to the reader.
\end{remark}

Let $\varphi\in B(H_1)\overline{\otimes} B(H_2)^{op}\overline{\otimes} B(H_3)$.
We will say that $\tau_\varphi$
is an {\bf $S^1$-operator multiplier} if it takes values into
the trace class
$S^1(H_1,H_3)$ and there exists a constant 
$D\geq 0$ such that
$$
\norm{\tau_\varphi(y,x)}_1\leq D\norm{x}_2\norm{y}_2,\qquad
x\in S^2(H_1,H_2),\ y\in S^2(H_2,H_3).
$$

Note that, by (\ref{2Sigma2}), $\tau_\varphi$ is an
$S^1$-operator multiplier when $\varphi$ is of the form $R\otimes S \otimes T$.
Consequently,  $\tau_\varphi$ is an
$S^1$-operator multiplier whenever $\varphi$ belongs to the algebraic
tensor product $B(H_1)\otimes B(H_2)\otimes B(H_3)$.

In this paper we will be mostly interested in {\bf completely bounded
$S^1$-operator multipliers}, that is, $S^1$-operator multipliers $\tau_\varphi$
such that $\widetilde{\tau_\varphi}$ is a completely bounded map from $\Gamma(H_1,H_2,H_3)$ into
$S^1(H_1,H_3)$. Note that the canonical inclusion $S^1(H_1,H_3)\subset B(H_1,H_3)$
is a complete contraction, hence
$$
CB(\Gamma(H_1,H_2,H_3), S^1(H_1,H_3))\,\subset\,
CB(\Gamma(H_1,H_2,H_3), B(H_1,H_3))\qquad\hbox{contractively.}
$$
It therefore follows from Proposition
\ref{2OM-B} that the space of all completely bounded
$S^1$-operator multipliers coincides with the space
$CB(\Gamma(H_1,H_2,H_3), S^1(H_1,H_3))$.
The following statement provides a characterization.

\begin{lemma}\label{2CB}
Let $u\colon S^2(H_2,H_3)\times S^2(H_1,H_2)\to S^1(H_1,H_3)$
be a bounded bilinear map and let
$D>0$ be a constant.
Then $\widetilde{u}\in CB(\Gamma(H_1,H_2,H_3), S^1(H_1,H_3))$ and
$\cbnorm{\widetilde{u}}\leq D$ if and only
if for any $n\geq 1$, for any $x_1,\ldots,x_n\in S^2(H_1,H_2)$ and for any
$y_1,\ldots,y_n\in S^2(H_2,H_3)$,
$$
\bignorm{\bigl[u(y_i,x_j)\bigr]_{1\leq i,j\leq n}}_{S^1(\ell^2_n(H_1),
\ell^2_n(H_3))}\,\leq D\,\Bigl(\sum_{j=1}^n
\norm{x_j}^2_2\Bigr)^{\frac12}\Bigl(\sum_{i=1}^n
\norm{y_i}^2_2\Bigr)^{\frac12}.
$$
\end{lemma}

\begin{proof}
For any $n\geq 1$, we use the classical notations
$R_n=\{\ell^2_n\}_r, C_n=\{\ell^2_n\}_c$ and
$S^1_n=S^1(\ell^2_n)$.

Consider $u$ as above and set
$$
d_n = \bignorm{I_{S^1_n}\otimes \widetilde{u} \colon
S_n^1\stackrel{\frown}{\otimes} \Gamma(H_1,H_2,H_3)
\longrightarrow S_n^1\stackrel{\frown}{\otimes}S^1(H_1,H_3)}
$$
for any $n\geq 1$.
By \cite[Lemma 1.7]{P1}, $\widetilde{u}\in
CB\bigl(\Gamma(H_1,H_2,H_3), S^1(H_1,H_3)\bigr)$ if and only if
the sequence $(d_n)_{n\geq 1}$ is bounded and in this case,
$\cbnorm{\widetilde{u}}=\sup_n d_n$.

By Proposition \ref{1Recap} (c),
$$
S_n^1\stackrel{\frown}{\otimes} \Gamma(H_1,H_2,H_3)
\,\simeq\,
R_n\stackrel{\frown}{\otimes} \{S^2(H_1,H_2)\}_r
\stackrel{\frown}{\otimes}\{S^2(H_2,H_3)\}_c\stackrel{\frown}{\otimes} C_n
$$
completely isometrically. Using Proposition \ref{1Recap} (b), this yields
$$
S_n^1\stackrel{\frown}{\otimes} \Gamma(H_1,H_2,H_3)
\,\simeq\,
\bigl\{\ell^2_n\stackrel{2}{\otimes}S^2(H_1,H_2)\bigr\}_r\stackrel{\frown}{\otimes}
\bigl\{\ell^2_n\stackrel{2}{\otimes}S^2(H_2,H_3)\bigr\}_c.
$$
Applying Remark \ref{1Equal}, we derive that
$$
S_n^1\stackrel{\frown}{\otimes} \Gamma(H_1,H_2,H_3)
\,\simeq\,
\bigl(\ell^2_n\stackrel{2}{\otimes}S^2(H_1,H_2)\bigr)\,\widehat{\otimes}\,
\bigl(\ell^2_n\stackrel{2}{\otimes}S^2(H_2,H_3)\bigr)
$$
isometrically.

Similarly,
\begin{align*}
S_n^1 \stackrel{\frown}{\otimes} S^1(H_1,H_3)
& \,\simeq\,
R_n \stackrel{\frown}{\otimes}  S^1(H_1,H_3)
\stackrel{\frown}{\otimes} C_n
\\
& \,\simeq\,
R_n\stackrel{\frown}{\otimes} \{\overline{H_1}\}_r
\stackrel{\frown}{\otimes} \{H_3\}_c\stackrel{\frown}{\otimes} C_n
\\
& \,\simeq\,
\bigl\{\ell^2_n\stackrel{2}{\otimes}\overline{H_1}\bigr\}_r
\stackrel{\frown}{\otimes}
\bigl\{\ell^2_n\stackrel{2}{\otimes} H_3\bigr\}_c
\\
& \,\simeq\,
S^1\bigl(\ell^2_n(H_1), \ell^2_n(H_3)\bigr)
\end{align*}
isometrically.
Hence a thorough look at these identifications shows that
$$
d_n = \sup\Bigl\{\bignorm{\bigl[u(y_i,x_j)\bigr]_{1\leq i,j\leq n}}_{S^1(\ell^2_n(H_1),
\ell^2_n(H_3))}\Bigr\},
$$
where the supremum runs over all
$$
(x_1,\ldots,x_n)\in \ell^2_n\stackrel{2}{\otimes}S^2(H_1,H_2)
\qquad\hbox{and}\qquad
(y_1,\ldots,y_n)\in\ell^2_n\stackrel{2}{\otimes}S^2(H_2,H_3)
$$
of norms less than
or equal to $1$. This yields the result.
\end{proof}

The next result, which should be compared to Proposition \ref{2OM-B},
provides a characterization of completely bounded
$S^1$-operator multipliers. Before stating it, we note
that we have $S^1(H_1)\stackrel{\frown}{\otimes}
S^1(H_3)\subset S^1(H_1)\stackrel{h}{\otimes}
S^1(H_3)$ completely contractively (see e.g. \cite[Theorem 9.2.1]{ER}).
Consequently
$$
CB\bigl(S^1(H_1)\stackrel{h}{\otimes}
S^1(H_3), B(H_2)^{op}\bigr)\,\subset\,
CB\bigl(S^1(H_1)\stackrel{\frown}{\otimes}
S^1(H_3), B(H_2)^{op}\bigr)
$$
contractively. Applying Lemma \ref{1Injective1} (c), and using
(\ref{1S1S1}) and (\ref{1w*h}), we deduce a
contractive embedding
$$
B(H_2)^{op}\overline{\otimes}\bigl(
B(H_1)\stackrel{w^*h}{\otimes}B(H_3)\bigr)\,\subset\,
B(H_1)\overline{\otimes} B(H_2)^{op}
\overline{\otimes} B(H_3).
$$

\begin{theorem}\label{2OM-S1}
Let $\varphi\in B(H_1)\overline{\otimes} B(H_2)^{op}\overline{\otimes} B(H_3)$.
Then $\tau_\varphi$ is a completely bounded $S^1$-operator
multiplier if and only if $\varphi$ belongs to
$B(H_2)^{op}\overline{\otimes}\bigl(
B(H_1)\stackrel{w^*h}{\otimes}B(H_3)\bigr)$. Further (\ref{2OM-B1})
restricts to a $w^*$-continuous
completely isometric identification
\begin{equation}\label{2Ident3}
B(H_2)^{op}\overline{\otimes}\bigl(
B(H_1)\stackrel{w^*h}{\otimes}B(H_3)\bigr)\,\simeq\,
CB\bigl(\Gamma(H_1,H_2,H_3), S^1(H_1,H_3)\bigr).
\end{equation}
\end{theorem}

\begin{proof} The scheme of proof is similar to the one of Proposition \ref{2OM-B}.
Recall (\ref{2Gamma}) from this proof.
On the one hand, using
commutativity of the operator space projective tensor product, we deduce a
completely isometric identification
$$
\Gamma(H_1,H_2,H_3)
\,\simeq\, \{H_2\}_r \stackrel{\frown}{\otimes}  \{\overline{H}_2\}_c
\stackrel{\frown}{\otimes} \{\overline{H}_1\}_r \stackrel{\frown}{\otimes}
\{H_3\}_c,
$$
and then, by Proposition \ref{1Recap} (c),
\begin{equation}\label{2Ident2}
\Gamma(H_1,H_2,H_3) \,\simeq\, S^1(\overline{H}_2)
\stackrel{\frown}{\otimes} S^1(H_1,H_3).
\end{equation}
On the other hand, it follows from (\ref{1REC}) and Proposition \ref{1Recap} (a) that
$$
S^1(H_1,H_3)
\stackrel{\frown}{\otimes} S^\infty(H_3,H_1)
\,\simeq\,\{\overline{H_1}\}_r\stackrel{h}{\otimes} S^\infty(H_3,H_1)\stackrel{h}{\otimes}\{H_3\}_c.
$$
Then using (\ref{1Compact}), we deduce
that
$$
S^1(H_1,H_3)
\stackrel{\frown}{\otimes} S^\infty(H_3,H_1)
\,\simeq\, \bigl(\{\overline{H_1}\}_r\stackrel{h}{\otimes}
\{H_1\}_c\bigr)
\stackrel{h}{\otimes} \bigl(\{\overline{H_3}\}_r\stackrel{h}{\otimes}\{H_3\}_c\bigr).
$$
Applying Proposition \ref{1Recap} (a) again together with (\ref{1RC}), we obtain that
$$
S^1(H_1,H_3)
\stackrel{\frown}{\otimes} S^\infty(H_3,H_1)
\,\simeq\,  S^1(H_1) \stackrel{h}{\otimes} S^1(H_3)
$$
completely isometrically.

Combining the last identification with (\ref{2Ident2}), we find
\begin{equation}\label{2Ident4}
\Gamma(H_1,H_2,H_3)\stackrel{\frown}{\otimes} S^\infty(H_3,H_1)\,\simeq\,
S^1(\overline{H}_2)
\stackrel{\frown}{\otimes}\bigl(S^1(H_1) \stackrel{h}{\otimes} S^1(H_3)\bigr).
\end{equation}

We now pass to duals. First by (\ref{1Duality4}) and (\ref{1S11}), we have
a $w^*$-continuous completely isometric identification
$$
\bigl(\Gamma(H_1,H_2,H_3)\stackrel{\frown}{\otimes} S^\infty(H_3,H_1)\bigr)^*
\,\simeq\,
CB\bigl(\Gamma(H_1,H_2,H_3), S^1(H_1,H_3)\bigr).
$$
Second by (\ref{1Duality4}) and Lemma \ref{1Opp}, we have
$w^*$-continuous completely isometric identifications
\begin{align*}
\bigl(S^1(\overline{H}_2)
\stackrel{\frown}{\otimes}\bigl(S^1(H_1) \stackrel{h}{\otimes} S^1(H_3)\bigr)\bigr)^*
\,&\simeq\,
CB\bigl(S^1(H_1) \stackrel{h}{\otimes} S^1(H_3),B(\overline{H_2})\bigr)\\
&\simeq\, CB\bigl(S^1(H_1) \stackrel{h}{\otimes} S^1(H_3),B(H_2)^{op}\bigr).
\end{align*}
Equivalently, by Lemma \ref{1Injective1} (c), we have
$$
\bigl(S^1(\overline{H}_2)
\stackrel{\frown}{\otimes}\bigl(S^1(H_1) \stackrel{h}{\otimes} S^1(H_3)\bigr)\bigr)^*
\,\simeq\,B(H_2)^{op}\overline{\otimes}\bigl(B(H_1)\stackrel{w^*h}{\otimes} B(H_3)\bigr).
$$
Thus (\ref{2Ident4}) yields a $w^*$-continuous, completely isometric onto mapping
$$
L\colon B(H_2)^{op}\overline{\otimes}\bigl(B(H_1)\stackrel{w^*h}{\otimes} B(H_3)\bigr)
\longrightarrow
CB\bigl(\Gamma(H_1,H_2,H_3), S^1(H_1,H_3)\bigr).
$$
Arguing as in the proof of Proposition \ref{2OM-B}, it now suffices to
show that
for any $R\in B(H_1)$, $S\in B(H_2)$ and
$T\in B(H_3)$,
$L(S\otimes R\otimes T)$ coincides with
$\widetilde{\tau}_{R\otimes S\otimes T}$.  Next, it suffices to show that
\begin{equation}\label{2Check2}
L(S\otimes R\otimes T)(y\otimes x)  = TySxR
\end{equation}
when $R,S,T$ are rank one and when $x\in S^2(H_1,H_2)$ and
$y\in S^2(H_2,H_3)$ are rank one.

We let $\xi_i,\eta_i,h_i,k_i\in H_i$ for $i=1,2,3$ and
consider
$R=h_1\otimes\overline{k_1}$, $S=h_2\otimes\overline{k_2}$,
$T=h_3\otimes\overline{k_3}$, $x=\overline{\xi_1}\otimes\eta_2$ and
$y= \overline{\xi_2}\otimes\eta_3$.
Then $y\otimes x\in \Gamma(H_1,H_2,H_3)$ corresponds
to $(\eta_2\otimes\overline{\xi_2})\otimes
(\overline{\xi_1}\otimes\eta_3)\in S^1(\overline{H}_2)\otimes
S^1(H_1,H_3)$ in the identification (\ref{2Ident2}). Hence
$y\otimes x\otimes(\eta_1\otimes\overline{\xi_3})$ regarded as an element
of
$\Gamma(H_1,H_2,H_3)\otimes S^\infty(H_3,H_1)$
corresponds
to
$$
(\eta_2\otimes\overline{\xi_2})\otimes
(\overline{\xi_1}\otimes\eta_1)\otimes(\overline{\xi_3}\otimes\eta_3)
\,\in\, S^1(\overline{H}_2)\otimes
S^1(H_1)\otimes S^1(H_3)
$$
in the identification (\ref{2Ident4}).
Since
$$
\widehat{S}\otimes R\otimes T=\overline{k_2}\otimes h_2\otimes
h_1\otimes\overline{k_1}\otimes h_3\otimes\overline{k_3}
\,\in\, B(\overline{H}_2)\otimes B(H_1)\otimes B(H_3),
$$
we then have
$$
\bigl\langle\bigl[
L(S\otimes R\otimes T)(y\otimes x )\bigr](\eta_1),\xi_3\bigr\rangle =
\langle \eta_2,k_2\rangle\langle h_2,\xi_2\rangle
\langle h_1,\xi_1\rangle\langle \eta_1,k_1\rangle
\langle h_3,\xi_3\rangle\langle \eta_3,k_3\rangle.
$$
By (\ref{2Trace}), the right hand side of this equality is equal
to $\langle TySxR(\eta_1),\xi_3\rangle$. This proves the identity (\ref{2Check2}),
and hence the result.
\end{proof}

\section{Module maps}\label{4MOD}

As in the previous section, we consider three Hilbert spaces $H_1,H_2,H_3$. We further
consider von Neumann subalgebras
$$
M_1\subset B(H_1),\qquad
M_2\subset B(H_2)\quad\hbox{and}\qquad
M_3\subset B(H_3)
$$
acting on these spaces. For $i=1,2,3$, we let $M_i'\subset B(H_i)$ be the commutant
of $M_i$.

Let $u\colon  S^2(H_2,H_3)\times S^2(H_1,H_2)\to B(H_1,H_3)$ be a bounded bilinear operator.
We say that $u$ is an $(M'_3,M'_2,M'_1)$-module map (or is $(M'_3,M'_2,M'_1)$-modular)
provided that
$$
u(Ty,x)=Tu(y,x),\qquad u(y,xR)=u(y,x)R\quad\hbox{and}\quad
u(yS,x)=u(y,Sx)
$$
for any $x\in S^2(H_1,H_2)$, $y\in S^2(H_2,H_3)$, $R\in M'_1$, $S\in M'_2$ and $T\in M'_3$.

It will be convenient to associate to $u$ the following $4$-linear bounded operators.
We define
\begin{equation}\label{3U11}
U_1\colon \overline{H_2}\times H_2\times \overline{H_3}\times H_3\longrightarrow B(H_1)
\end{equation}
by
\begin{equation}\label{3U12}
\bigl\langle \bigl[U_1(\overline{\xi_2},\eta_2,\overline{\xi_3},\eta_3)\bigr]
(\eta_1),\xi_1\bigr\rangle
\,=\, \bigl\langle \bigl[u(\overline{\xi_2}\otimes\eta_3, \overline{\xi_1}\otimes\eta_2)\bigr]
(\eta_1),\xi_3\bigr\rangle
\end{equation}
for any $\xi_1,\eta_1\in H_1$, $\xi_2,\eta_2\in H_2$ and $\xi_3,\eta_3\in H_3$. Likewise we define
$$
U_2\colon\overline{H_1}\times H_1\times \overline{H_3}\times H_3\to B(H_2)
\qquad\hbox{and}\qquad
U_3\colon\overline{H_1}\times H_1\times \overline{H_2}\times H_2\to B(H_3)
$$
by
\begin{align*}
\bigl\langle \bigl[U_2(\overline{\xi_1},\eta_1,\overline{\xi_3},\eta_3)\bigr]
(\eta_2),\xi_2\bigr\rangle
\,&=\, \bigl\langle \bigl[u(\overline{\xi_2}\otimes\eta_3, \overline{\xi_1}\otimes\eta_2)\bigr]
(\eta_1),\xi_3\bigr\rangle\\
\bigl\langle \bigl[U_3(\overline{\xi_1},\eta_1,\overline{\xi_2},\eta_2)\bigr]
(\eta_3),\xi_3\bigr\rangle
\,&=\, \bigl\langle \bigl[u(\overline{\xi_2}\otimes\eta_3, \overline{\xi_1}\otimes\eta_2)\bigr]
(\eta_1),\xi_3\bigr\rangle.
\end{align*}

\begin{lemma}\label{3LemMod}
Let $u\in B_2\bigl(S^2(H_2,H_3)\times S^2(H_1,H_2), B(H_1,H_3)\bigr)$.
Then $u$ is an $(M'_3,M'_2,M'_1)$-module map if and only if
for any $i=1,2,3$, $U_i$ is valued in $M_i$.
\end{lemma}

\begin{proof}
Let $R\in B(H_1)$. For any $\eta_1,\xi_1\in H_1$,
$\eta_2,\xi_2\in H_2$ and $\eta_3,\xi_3\in H_3$, we have
$$
\bigl\langle \bigl[u(\overline{\xi_2}\otimes\eta_3, \overline{\xi_1}\otimes\eta_2)\bigr]
R(\eta_1),\xi_3\bigr\rangle\,
=\,
\bigl\langle \bigl[U_1(\overline{\xi_2},\eta_2,\overline{\xi_3},\eta_3)\bigr]
R(\eta_1),\xi_1\bigr\rangle.
$$
Further $(\overline{\xi_1}\otimes\eta_2)R= \overline{R^*(\xi_1)}\otimes\eta_2$, hence
\begin{align*}
\bigl\langle \bigl[u(\overline{\xi_2}\otimes\eta_3, (\overline{\xi_1}\otimes\eta_2)R)\bigr]
(\eta_1),\xi_3\bigr\rangle\,
& =\,\bigl\langle \bigl[U_1(\overline{\xi_2},\eta_2,\overline{\xi_3},\eta_3)\bigr]
(\eta_1),R^*(\xi_1)\bigr\rangle\\
& =\,\bigl\langle R\bigl[U_1(\overline{\xi_2},\eta_2,\overline{\xi_3},\eta_3)\bigr]
(\eta_1),\xi_1\bigr\rangle.
\end{align*}

Since $\overline{H_1}\otimes H_2$ and $\overline{H_2}\otimes H_3$
are dense in $S^2(H_1,H_2)$ and
$S^2(H_2,H_3)$, respectively,
we deduce that $u(y,xR)=u(y,x)R$ for any $x\in S^2(H_1,H_2)$ and
any $y\in S^2(H_2,H_3)$ if and only if $R$ commutes with
$U_1(\overline{\xi_2},\eta_2,\overline{\xi_3},\eta_3)$ for
any $\xi_2,\eta_2\in H_2$ and $\xi_3,\eta_3\in H_3$.

Consequently
$u$ is $(\Cdb,\Cdb,M_1')$-modular if and only if
the range of $U_1$ commutes with $M_1'$. By the
Bicommutant Theorem, this means that
$u$ is $(\Cdb,\Cdb,M_1')$-modular if and only if
$U_1$ is valued in $M_1$.

Likewise $u$ is $(\Cdb,M_2',\Cdb)$-modular (resp.
$(M_3',\Cdb,\Cdb)$-modular)
if and only if
$U_2$ is valued in $M_2$ (resp. $U_3$ is valued in $M_3$).
This proves the result.
\end{proof}

\begin{corollary}\label{3Mod-B}
Let $\varphi\in B(H_1)\overline{\otimes} B(H_2)^{op}\overline{\otimes}
B(H_3)$. Then $\tau_\varphi$ is $(M_3',M_2',M_1')$-modular if and only
if $\varphi\in M_1\overline{\otimes} M_2^{op}\overline{\otimes}
M_3$. 
This provides (as a restriction of (\ref{2OM-B1}))
a $w^*$-continuous
completely isometric identification
$$
M_1\overline{\otimes} M_2^{op}\overline{\otimes} M_3
\,\simeq\, CB_{(M'_3,M'_2,M'_1)}
\bigl(\Gamma(H_1,H_2,H_3), B(H_1,H_3)\bigr),
$$
where the right-hand side denotes the subspace 
of $CB\bigl(\Gamma(H_1,H_2,H_3), B^1(H_1,H_3)\bigr)$
of all completely bounded maps $\widetilde{u}$ such that
$u$ is an $(M'_3,M'_2,M'_1)$-module map.
\end{corollary}

\begin{proof}
Consider the duality relation
$$
B(H_2)^{op}\overline{\otimes} B(H_3) = 
\bigl(S^1(H_2)^{op} \stackrel{\frown}{\otimes} S^1(H_3)\bigr)^*
$$
provided by (\ref{1S1S1}).
We claim that in the space $S^1(H_2)^{op} \stackrel{\frown}{\otimes} S^1(H_3)$, we have equality
\begin{equation}\label{3perp}
\bigl(M_{2}^{op}\overline{\otimes} M_3\bigr)_{\perp}
\,=\,\overline{(M_{2\perp}^{op}\otimes S^1(H_3) + S^1(H_2)^{op}\otimes M_{3\perp})}.
\end{equation}
Indeed let $z\in B(H_2)^{op}\overline{\otimes} B(H_3)$ and let
$z'\colon S^1(H_3)\to B(H_2)^{op}$ and
$z''\colon S^1(H_2)^{op} \to B(H_3)$ 
be associated with $z$ (see Lemma \ref{1Slice}).
Then $z\in \bigl(M_{2\perp}^{op}\otimes S^1(H_3)\bigr)^{\perp}$ if and only if
$z'$ is valued in $M_{2}^{op}$, whereas $z\in \bigl(S^1(H_2)^{op}\otimes M_{3\perp}\bigr)^{\perp}$ if and only if
$z''$ is valued in $M_{3}$. Consequently,
$z$ belongs to the orthogonal of $M_{2\perp}^{op}\otimes S^1(H_3) + S^1(H_2)^{op}\otimes M_{3\perp}$
if and only if $z'$ is valued in $M_{2}^{op}$
and $z''$ is valued in $M_{3}$. In turn this is equivalent to $z'\in CB(M_{3*}, M_{2}^{op})$.
Applying Lemma \ref{1Injective1} (a), we deduce that
the orthogonal of $M_{2\perp}^{op}\otimes S^1(H_3) + S^1(H_2)^{op}\otimes M_{3\perp}$
is equal to $M_{2}^{op}\overline{\otimes} M_3$. The claim (\ref{3perp})
follows at once.

Let $\varphi\in B(H_1)\overline{\otimes} B(H_2)^{op}\overline{\otimes}
B(H_3)$. Using Lemma \ref{1Slice}, we may associate 3 completely bounded operators
\begin{align*}
\varphi^1 &\colon S^1(H_2)^{op}\stackrel{\frown}{\otimes} S^1(H_3)\longrightarrow
B(H_1),\\
\varphi^2 &\colon S^1(H_1)\stackrel{\frown}{\otimes} S^1(H_3)\longrightarrow
B(H_2)^{op},\\
\varphi^3 &\colon S^1(H_1)\stackrel{\frown}{\otimes} S^1(H_2)^{op}\longrightarrow
B(H_3)
\end{align*}
to $\varphi$.
According to Lemma \ref{1Injective1} (a), $\varphi$ belongs
to $M_1\overline{\otimes} M_2^{op}\overline{\otimes}
M_3$ if and only if $\varphi^1$ is valued in $M_1$ and
$\varphi^1$ vanishes on $(M_2^{op}\overline{\otimes}
M_3)_{\perp}$. By (\ref{3perp}),
$\varphi^1$ vanishes on $(M_2^{op}\overline{\otimes}
M_3)_{\perp}$ if and only if it both vanishes on
$M_{2\perp}^{op}\otimes S^1(H_3)$ and $S^1(H_2)^{op}\otimes M_{3\perp}$.
A quick look at the definitions of $\varphi^1, \varphi^2,\varphi^3$
reveals that $\varphi^1$ vanishes on
$M_{2\perp}^{op}\otimes S^1(H_3)$ if and only if $\varphi^2$ is valued in
$M^{op}_{2}$ and that $\varphi^1$ vanishes on $S^1(H_2)^{op}\otimes M_{3\perp}$
if and only if $\varphi^3$ is valued in
$M_{3}$. Altogether we obtain that $\varphi$ belongs
to $M_1\overline{\otimes} M_2^{op}\overline{\otimes}
M_3$ if and only if  $\varphi^1$ is valued in $M_1$,
$\varphi^2$ is valued in $M_2^{op}$ and
$\varphi^3$ is valued in $M_3$.

Let $u=\tau_\varphi$.
It follows from Remark \ref{2Rk}
that for any $\eta_2,\xi_2\in H_2$ and $\eta_3,\xi_3\in H_3$, we have
$$
\varphi^1(\eta_2\otimes \overline{\xi_2}\otimes \overline{\xi_3}\otimes
\eta_3) = U_1(\overline{\xi_2},\eta_2,\overline{\xi_3},\eta_3),
$$
where $U_1$ is defined by (\ref{3U11}) and (\ref{3U12}).
Thus $\varphi^1$ is valued in
$M_1$ if and only if $U_1$ is valued in $M_1$. Likewise
$\varphi^2$ is valued in
$M_2^{op}$ if and only if $U_2$ is valued in $M_2$ and
$\varphi^3$ is valued in
$M_3$ if and only if $U_3$ is valued in $M_3$.
By Lemma \ref{3LemMod} we deduce that
$u$ is $(M_1',M_2',M_3')$-modular if and only
if $\varphi\in M_1\overline{\otimes} M_2^{op}\overline{\otimes} M_3$.
\end{proof}

We now turn to the study of modular completely bounded $S^1$-multipliers.
We let
$$
CB_{(M'_3,M'_2,M'_1)}\bigl(\Gamma(H_1,H_2,H_3), S^1(H_1,H_3)\bigr)
$$
denote the subspace of $CB\bigl(\Gamma(H_1,H_2,H_3), S^1(H_1,H_3)\bigr)$
of all completely bounded maps $\widetilde{u}$ such that
$u$ is an $(M'_3,M'_2,M'_1)$-module map.

According to (\ref{1Embed}) and  (\ref{1w*h}),
$M_1\stackrel{w^*h}{\otimes} M_3$ can be regarded as
a $w^*$-closed subspace of the dual operator space
$B(H_1)\stackrel{w^*h}{\otimes}B(H_3)$. Consequently,
$M_2^{op}\overline{\otimes}\bigl(M_1\stackrel{w^*h}{\otimes} M_3\bigr)$
can be regarded as
a $w^*$-closed subspace of the dual operator space
$B(H_2)^{op}\overline{\otimes}\bigl(B(H_1)\stackrel{w^*h}{\otimes}B(H_3)\bigr)$.
The next statement is a continuation of Theorem \ref{2OM-S1}.

\begin{theorem}\label{3OM-S1}
Assume that $M_2$ is injective.
\begin{itemize}
\item [(a)] Let $\varphi\in B(H_2)^{op}\overline{\otimes}\bigl(B(H_1)\stackrel{w^*h}{\otimes}B(H_3)\bigr)$.
Then $\varphi$ belongs to $M_2^{op}\overline{\otimes}\bigl(M_1\stackrel{w^*h}{\otimes} M_3\bigr)$
if and only if $\tau_\varphi$ is $(M'_3,M'_2,M'_1)$-modular.

\smallskip
\item [(b)]
The identification (\ref{2Ident3}) restricts to
\begin{equation}\label{3Ident1}
M_2^{op}\overline{\otimes}\bigl(
M_1\stackrel{w^*h}{\otimes}M_3\bigr)\,\simeq \,
CB_{(M'_3,M'_2,M'_1)}\bigl(\Gamma(H_1,H_2,H_3), S^1(H_1,H_3)\bigr).
\end{equation}
\end{itemize}
\end{theorem}

\begin{proof}
Clearly (b) is a consequence of (a) so we only treat this first item.

Let $\varphi\in B(H_2)^{op}\overline{\otimes}\bigl(B(H_1)\stackrel{w^*h}{\otimes}B(H_3)\bigr)$.
Let
$$
\sigma\colon S^1(H_1)\stackrel{h}{\otimes}S^1(H_3)\longrightarrow  B(H_2)^{op}
$$
be corresponding to $\varphi$ in the identification provided by Lemma \ref{1Injective1} (c).
Then let
$$
\rho\colon S^1(H_2)^{op} \longrightarrow B(H_1)\stackrel{w^*h}{\otimes} B(H_3)
$$
be the restriction of the adjoint of $\sigma$
to $S^1(H_2)^{op}$.

We assumed that $M_2$ is injective. It therefore
follows from Lemma \ref{1Injective1} (b) that
$\varphi\in M_2^{op}\overline{\otimes}\bigl(M_1\stackrel{w^*h}{\otimes} M_3\bigr)$
if and only if
\begin{equation}\label{2Fubini1}
\sigma\bigl(S^1(H_1)\stackrel{h}{\otimes}S^1(H_3)\bigr)
\,\subset\, M_2^{op}
\end{equation}
and
\begin{equation}\label{2Fubini2}
\rho\bigl(S^1(H_2)^{op}\bigr)\,\subset\,M_1\stackrel{w^*h}{\otimes} M_3.
\end{equation}

Let  $u=\tau_\varphi$.
We will now show that $u$ is an $(M'_3,M'_2,M'_1)$-module map if and only if
(\ref{2Fubini1}) and (\ref{2Fubini2}) hold true.

First we observe that for any $\xi_1,\eta_1\in H_1$ and $\xi_3,\eta_3\in H_3$,
$$
\sigma\bigl((\overline{\xi_1}\otimes\eta_1)\otimes(\overline{\xi_3}\otimes\eta_3)\bigr)
\,=\, U_2(\overline{\xi_1},\eta_1,\overline{\xi_3},\eta_3).
$$
Indeed, this follows from Remark \ref{2Rk} and the definition of $U_2$.
Since
$\overline{H_1}\otimes H_1$ and $\overline{H_3}\otimes H_3$
are dense in $S^1(H_1)$ and $S^1(H_3)$, respectively,
we deduce that (\ref{2Fubini1}) holds true if and only if
$U_2$ is valued in $M_2$.

For any $v\in S^1(H_2)^{op}$, we may regard $\rho(v)$ as
an element of  $\bigl(S^1(H_1)\stackrel{h}{\otimes} S^1(H_3)\bigr)^*$.
Then following the notation in Lemma \ref{1Slice-H}, we let
$$
[\rho(v)]'\colon S^1(H_3)\longrightarrow B(H_1)
\qquad\hbox{and}\qquad
[\rho(v)]''\colon S^1(H_1)\longrightarrow B(H_3)
$$
be the bounded linear maps associated to $\rho(v)$.

For any $\xi_2,\eta_2\in H_2$ and $\xi_3,\eta_3\in H_3$, we have
$$
\bigl[\rho(\eta_2\otimes\overline{\xi_2})\bigr]'(\overline{\xi_3}\otimes\eta_3) =
U_1(\overline{\xi_2},\eta_2, \overline{\xi_3},\eta_3).
$$
Indeed this follows again from Remark \ref{2Rk}.
Since
$H_2\otimes \overline{H_2}$ and $\overline{H_3}\otimes H_3$
are dense in $S^1(H_2)^{op}$ and $S^1(H_3)$, respectively,
we deduce that $[\rho(v)]'$ maps
$S^1(H_3)$ into $M_1$ for any  $v\in S^1(H_2)^{op}$ if and only
if $U_1$ is valued in $M_1$. Likewise, $[\rho(v)]''$ maps
$S^1(H_1)$ into $M_3$ for any  $v\in S^1(H_2)^{op}$ if and only
if $U_3$ is valued in $M_3$. Applying Lemma \ref{1Slice-H}, we
deduce  (\ref{2Fubini2}) holds true if and only if
$U_1$ is valued in $M_1$ and $U_3$ is valued in $M_3$.

Altogether we have that (\ref{2Fubini1}) and (\ref{2Fubini2}) both
hold true if and only if for any $i=1,2,3$,
$U_i$ is valued in $M_i$. According to Lemma \ref{3LemMod}, this is equivalent to
$u=\tau_\varphi$ being
$(M'_3,M'_2,M'_1)$-modular.	
\end{proof}

\section{The Sinclair-Smith factorization theorem}\label{5SS}

Let $I$ be an index set, and consider the Hilbertian operator spaces
$$
C_I = \{\ell^2_I\}_c
\qquad\hbox{and}\qquad
R_I = \{\ell^2_I\}_r.
$$
For any operator space $G$, we set
$$
C_I^w(G^*) = C_I\overline{\otimes} G^*
\qquad\hbox{and}\qquad
R_I^w(G^*) = R_I\overline{\otimes} G^*.
$$
This notation is taken from \cite[1.2.26--1.2.29]{BLM}, to which we refer for more
information.

We recall that $C_I^w(G^*)$ can be equivalently defined as the space of all families
$(x_i)_{i\in I}$ of elements of $G^*$ such that
the sums $\sum_{i\in J} x^*_ix_i$, for finite $J\subset I$, are uniformly bounded.
Likewise, $R_I^w(G^*)$ is equal to the space of all families
$(y_i)_{i\in I}$ of elements of $G^*$ such that
the sums $\sum_{i\in J} y_iy_i^*$, for finite $J\subset I$, are uniformly bounded.

Assume that $G^*=M$ is a von Neumann algebra, and consider $(x_i)_{i\in I}\in
C_I^w(M)$ and $(y_i)_{i\in I}\in
R_I^w(M)$. Then the family $(y_i x_i)_{i\in I}$ is summable in the $w^*$-topology of
$M$ and we let
\begin{equation}\label{4Sum}
\sum_{i\in I} y_i x_i\ \in M
\end{equation}
denote its sum.

We note the obvious fact that for any $x_i\in M, i\in I$,
$(x_i)_{i\in I}$ belongs to
$R_I^w(M)$ if and only if $(x_i^*)_{i\in I}$ belongs to
$C_I^w(M)$. In this case we set
$$
\bigl[(x_i)_{i\in I}\bigr]^* \,=\, (x_i^*)_{i\in I}.
$$

\begin{lemma}\label{4CI-CB} Let $E,G$ be operator spaces and let $I$ be an index set.
For any $\alpha=(\alpha_i)_{i\in I}\in C_I^w\bigl(CB(E,G^*)\bigr)$,
the (well-defined) operator $\widehat{\alpha}\colon E\to C_I^w(G^*)$,
$\widehat{\alpha}(x) =(\alpha_i(x))_{i\in I}$,
is completely bounded
and the mapping $\alpha\mapsto \widehat{\alpha}$ induces a
$w^*$-continuous completely isometric identification
$$
C_I^w\bigl(CB(E,G^*)\bigr)\,\simeq\, CB\bigl(E, C_I^w(G^*)\bigr).
$$
Likewise we have
$$
R_I^w\bigl(CB(E,G^*)\bigr)\,\simeq\, CB\bigl(E, R_I^w(G^*)\bigr).
$$
\end{lemma}

\begin{proof}
According to Lemma \ref{1Injective1} (c) and (\ref{1Duality4}),
$C_I^w(Z^*)\simeq (R_I\stackrel{\frown}{\otimes} Z)^*$ for any
operator space $Z$. Applying this identification, first with $Z=E\stackrel{\frown}{\otimes} G$
and then with
$Z=G$, we obtain that
\begin{align*}
C_I^w\bigl(CB(E,G^*)\bigr)\,
&\simeq\,
C_I^w\bigl((E\stackrel{\frown}{\otimes} G)^*\bigr)\quad\hbox{by (\ref{1Duality4})}\\
&\simeq\,
(R_I\stackrel{\frown}{\otimes} E\stackrel{\frown}{\otimes} G)^*\\
&\simeq\, CB\bigl(E, (R_I\stackrel{\frown}{\otimes} G)^*\bigr)\quad\hbox{by (\ref{1Duality4})}\\
&\simeq\, CB\bigl(E, C_I^w(G^*)\bigr).
\end{align*}
A straightforward verification reveals that this identification is implemented
by $\widehat{\alpha}$.
This yields the first part of the lemma. The proof of the second part is identical.
\end{proof}

We can now state the Sinclair-Smith factorization theorem, which will be use
in the next section.

\begin{theorem}\label{4SS} (\cite{SS})
Let $E,F$ be operator spaces, let $M$ be an injective von Neumann algebra
and let $w\colon F\stackrel{h}{\otimes} E\to M$ be a completely bounded
map. Then there exist an index set $I$ and two families
$$
\alpha=(\alpha_i)_{i\in I}\in C_I^w\bigl(CB(E,M)\bigr)
\qquad\hbox{and}\qquad
\beta=(\beta_i)_{i\in I}\in R_I^w\bigl(CB(F,M)\bigr)
$$
such that $\cbnorm{\alpha}\cbnorm{\beta} = \cbnorm{w}$ and
$$
w(y\otimes x)\,=\,\sum_{i\in I} \beta_i(y)\alpha_i(x),
\qquad x\in E,\, y\in F.
$$
\end{theorem}

In the rest of this section, we give a
new (shorter) proof of Theorem \ref{4SS}
based on Hilbert $C^*$-modules.

In the following we give the necessary background on
Hilbert $C^*$-modules. Let $M$ be a $C^*$-algebra. Recall that a
pre-Hilbert $M$-module is a right $M$-module $\X$ equipped with a map
$\langle\,\cdotp,\,\cdotp\rangle\colon\X\times \X\to M$ (called an $M$-valued
inner product) satisfying the following properties:
\begin{itemize}
\item $\langle s,s\rangle\geq 0$ for every $s\in \X$;
\item $\langle s,s\rangle=0$ if and only if $s=0$;
\item $\langle s,t\rangle=\langle t,s\rangle^*$  for every $s,t\in \X$;
\item $\langle s, t_1m_1+t_2m_2\rangle=\langle s,t_1\rangle m_1+\langle s,t_2\rangle m_2$
for every $s,\,t_1,\,t_2\in \X$ and $m_1,\, m_2\in M$.
\end{itemize}
In this setting, the map
$\|\cdot\|\colon\X\to\mathbb R^+$, defined by
$$
\|s\|=\|\langle s,s\rangle\|^{1/2}, \qquad s\in \X,
$$
is a norm on $\X$. A pre-Hilbert $M$-module which is complete
with respect to its norm is said to be a Hilbert $M$-module.

By \cite{B} (see also \cite[8.2.1]{BLM}),
a Hilbert $M$-module $\X$ has a canonical operator
space structure obtained by letting, for any $n\geq 1$,
$$
\bignorm{(s_{ij})_{i,j}} \, =\, \Biggnorm{\Biggl(\sum_{k=1}^n\langle s_{ki}, s_{kj}
\rangle\Biggr)_{i,j}}_{M_n(M)}^{1/2}, \qquad (s_{ij})_{i,j}\in M_n(\X).
$$

A morphism between two Hilbert $M$-modules $\X_1$ and $\X_2$
is a bounded $M$-module map $u\colon \X_1\to \X_2$. A unitary
isomorphism $u\colon \X_1\to \X_2$ is an isomorphism preserving
the $M$-valued inner products. Any such map is a complete isometry
(see e.g. \cite[Proposition 8.2.2]{BLM}).

Assume now that $M$ is a von Neumann algebra. As a basic
example, we recall that whenever $p\in M$ is a projection,
then the subspace $pM$ of $M$ is a Hilbert $M$-module, when equipped with
multiplication on the right as the $M$-module action, and
with the $M$-valued inner product $\langle x,y\rangle = x^*y$, for $x,y\in pM$.

We recall the construction of the ultraweak direct sum Hilbert $M$-module.
Let $I$ be an index set and let $\{\X_i\, :\, i\in I\}$
be a collection of Hilbert $M$-modules indexed by $I$. We let
$\langle\,\cdotp,\,\cdotp\rangle_i$ denote the $M$-valued inner product of $\X_i$,
for any $i\in I$. Let $\X$ be the set of all families $s=(s_i)_{i\in I}$,
with $s_i\in\X_i$, such that
the sums $\sum_{i\in J}\langle s_i,s_i\rangle_i$, for finite $J\subset I$, are uniformly bounded.
Since $\langle s_i,s_i\rangle_i\geq 0$ for each $i\in I$,
the family $(\langle s_i,s_i\rangle_i)_{i\in I}$ is
then summable in the $w^*$-topology of
$M$. Using polarization identity, it is easy to deduce that for any
$s=(s_i)_{i\in I}$ and any $t=(t_i)_{i\in I}$ in $\X$,
the family $(\langle s_i,t_i\rangle_i)_{i\in I}$ is
summable in the $w^*$-topology of $M$. Then one defines
$$
\langle s,t\rangle\,=\, \sum_{i\in I} \langle s_i,t_i\rangle_i.
$$
It turns out that $\X$ is a right $M$-module
for the action $(s_i)_{i\in I}\cdot m = (s_i m)_{i\in I}$, and that
equipped with $\langle \,\cdotp,\,\cdotp\rangle$,
$\X$ is a Hilbert $M$-module. The latter is called the ultraweak direct sum
of $\{\X_i\, :\, i\in I\}$  and it is denoted by
$$
\X\,=\,\oplus_{i\in I} \X_i.
$$
See e.g. \cite[8.5.26]{BLM} for more on this construction.

Let $I$ be an index set, consider $C_I^w(M)$ as a right $M$-module is the obvious way. For any
$(s_i)_{i\in I}$ and $(t_i)_{i\in I}$ in $C_I^w(M)$
set
$$
\langle (s_i)_{i\in I}, (t_i)_{i\in I}\rangle=\sum_{i\in I} s_i^* t_i\,,
$$
where this sum is defined by (\ref{4Sum}). This is an
$M$-valued inner product, which makes $C_I^w(M)$ a Hilbert $M$-module. Moreover the canonical
operator space structure of $C_I^w(M)$ as a Hilbert $M$-module coincides with
the one given by writing $C_I^w(M)=C_I\overline{\otimes} M$, see
\cite[8.2.3]{BLM}. Further we clearly have
$$
C_I^w(M)\,\simeq\,\oplus_{i\in I} M\qquad \hbox{as Hilbert } M\hbox{-modules}.
$$

\begin{proof}[Proof of Theorem \ref{4SS}]
Assume that $M\subset B(K)$ for some Hilbert space $K$.
Let $w\colon F\stackrel{h}{\otimes} E\to M$ be a completely bounded
map. By the Christensen-Sinclair factorization theorem
(see e.g. \cite[Theorem 9.4.4]{ER}),
there exist a Hilbert space $\H$ and two completely bounded maps
$$
a\colon E\to B(K,\H)
\qquad\hbox{and}\qquad
b\colon F\to  B(\H,K)
$$
such that $\cbnorm{a}\cbnorm{b}=\cbnorm{w}$
and $w(y\otimes x)=b(y)a(x)$ for any $x\in E$ and any $y\in F$.

Since $M$ is injective, there exists a unital completely positive projection
$$
\Psi \colon B(K)\longrightarrow M.
$$
As $\Psi$ is valued in $M$, we then
have
\begin{equation}\label{4Facto1}
w(y\otimes x)=\Psi\bigl(b(y)a(x)\bigr),
\qquad x\in E, \ y\in F.
\end{equation}

We introduce
$$
C\,=\,\bigl\{T\in B(K,\H)\,:\,  \Psi(T^*T)=0\bigr\}.
$$
For any $k\in K$, $(T,S)\mapsto\langle\Psi(T^*S)k,k\rangle$
is a nonnegative sesquilinear form
on $B(K,\H)$, which vanishes on $\{(T,T)\, :\, T\in C\}$. This implies
(by the Cauchy-Schwarz inequality) that $\langle\Psi(T^*S )k,k\rangle=0$
for any $T\in C$ and any $S\in B(K,\H)$. Consequently,
$$
C\,=\, \bigl\{T\in  B(K,\H)\, :\,\Psi(T^*S)=0 \text{ for any } S\in B(K,\H)\bigr\}.
$$
In particular $C$ is a subspace of $B(K,\H)$.
Moreover  $\Psi$ is an $M$-bimodule map by \cite{To}, hence
$$
\Psi((Tm)^*(Tm))=\Psi(m^*T^*Tm)=m^*\Psi(T^*T)m,\qquad m\in M,\, T\in B(K,\H).
$$
Consequently, $C$ is invariant under right multiplication by elements of $M$.

Let $N=B(K,\H)/C$ and let $q\colon B(K,\H)\to N$
be the quotient map. The $M$-invariance of $C$ allows to
define a right $M$-module action on $N$ by
$$
q(T)\cdot m = q(Tm),\qquad m\in M,\ T\in B(K,\H).
$$
For any $S,\, T\in B(K,\H)$, set
$$
\langle q(T),q(S)\rangle_{N}=\Psi(T^*S).
$$
Then $\langle\,\cdotp,\,\cdotp\rangle_{N}$ is a well-defined, $M$-valued inner product on $N$,
and hence $N$ is a pre-Hilbert $M$-module.
For convenience, we keep the notation $N$ to denote its completion, which is a Hilbert $M$-module.
The factorization property
(\ref{4Facto1}) can now be rephrased as
\begin{equation}\label{4Facto2}
w(y\otimes x)=\bigl\langle q(b(y)^*),q(a(x))\bigr\rangle_N,
\qquad x\in E, \ y\in F.
\end{equation}

Recall from Paschke's fundamental paper \cite{Pa}
that the dual of $N$ (in the Hilbert $M$-module sense)
is the space
$$
N' = \bigl\{\phi \colon   N \to   M \, : \, \phi \mbox{ is a bounded } M\mbox{-module map}\bigr\}.
$$
Equip $N'$ with the linear structure obtained with usual addition of maps and scalar
multiplication given by $(\lambda\cdot\phi)(t) =\overline{\lambda}\phi(t)$  for any
$\phi\in N'$, $\lambda\in\Cdb$, and $t\in N$.
Then $N'$ is a right $M$-module for the action
given by
$$
(\phi\cdot m)(t) = m^*\phi(t),
\qquad \phi\in N',\  m\in M,\ t\in N.
$$
Let $\kappa\colon N\to N'$ be defined by
$\kappa(s)\colon t\in N \mapsto \langle s,t\rangle\in M$.
Then $\kappa$ is a linear map. By \cite[Theorem 3.2]{Pa},
there exists
an $M$-valued inner product $\langle\,\cdotp ,
\,\cdotp\rangle_{N'}$ on $N'$ such
that
\begin{equation}\label{4Kappa1}
\langle \kappa(s),\kappa(t)\rangle_{N'}\, =\, \langle s,t\rangle_N,
\qquad s,t\in N,
\end{equation}
and such that $N'$ is selfdual (see \cite[Section 3]{Pa}
for the definition). Then
by \cite[Theorem 3.12]{Pa}, $N'$ is unitarily
isomorphic to an ultraweak  direct sum  $\displaystyle{\oplus_{i\in I}} p_i M$,
where $(p_i)_{i\in I}$ is a family of non-zero projections in $M$.
Summarizing, we then have
\begin{equation}\label{4Kappa2}
N\,\stackrel{\kappa}{\hookrightarrow}\,
N'\,\simeq\, \oplus_i p_i M \,\subset\, \oplus_i M  \,\simeq C_I^w(M).
\end{equation}
Note that by (\ref{4Kappa1}), $\kappa$ is a complete isometry.

We claim that the quotient map
$q \colon B(K,\H)\to  N$ is completely contractive, when $N$
is equipped with its Hilbert $M$-module operator space structure.
Indeed, $\Psi$ is completely contractive
hence, for any $(S_{ij})_{i,j}\in M_n(B(K,\H))$, we have
\begin{align*}
\bignorm{\bigl(q(S_{ij})\bigr)_{i,j}}^2_{M_n(N)}\,
& =\,
\Biggnorm{\left(\sum_{k=1}^n \Psi(S_{ki}^* S_{kj}) \right)_{i,j}}_{M_n(M)}\\
& \leq \,
\Biggnorm{\left(\sum_{k=1}^n S_{ki}^* S_{kj}\right)_{i,j}}_{M_n(B(K))}\\
& = \,
\bignorm{\bigl((S_{ij})_{i,j}^*(S_{ij})_{i,j}\bigr)}_{M_n(B(K))}\\
&=\, \bignorm{(S_{ij})_{i,j}}^2_{M_n(B(K,\tiny{\H}))}.
\end{align*}

Using (\ref{4Kappa2}), we define $\alpha\colon E\to C_I^w(M)$
by $\alpha(x)=\kappa\bigl(q(a(x))\bigr)$. It follows from above that
$\alpha$ is completely bounded, with $\cbnorm{\alpha}\leq\cbnorm{a}$.
Likewise we  define $\beta\colon F\to R_I^w(M)$
by $\beta(y)=\bigl[\kappa\bigl(q(b(y)^*)\bigr)]^*$. Then
$\beta$ is completely bounded, with $\cbnorm{\beta}\leq\cbnorm{b}$.
Consequently, $\cbnorm{\alpha}\cbnorm{\beta} \leq \cbnorm{w}$.

In accordance with Lemma \ref{4CI-CB}, let
$(\alpha_i)_{i\in I}\in C^w_I\bigl(CB(E,M)\bigr)$
and $(\beta_i)_{i\in I}\in R^w_I\bigl(CB(F,M)\bigr)$ be corresponding
to $\alpha$ and $\beta$, respectively.
Then by (\ref{4Facto2}) and (\ref{4Kappa2}), we have
$$
w(y\otimes x)\,=\,\langle\beta(y)^*,\alpha(x)\rangle_{N'}
\,=\, \bigl\langle (\beta_i(y)^*)_{i\in I},(\alpha_i(x))_{i\in I}
\bigr\rangle_{C_I^w(M)}\,
=\,
\sum_{i\in I} \beta_i(y)\alpha_i(x)
$$
for any $x\in E$ and $y\in F$.

Once this identity is established, the inequality $\cbnorm{w}\leq
\cbnorm{\alpha}\cbnorm{\beta}$ is a classical fact.
\end{proof}

\section{Factorization of modular operators}\label{6Facto}

Consider $H_1,H_2,H_3$ and $M_1,M_2,M_3$, $M_i\subset B(H_i)$, as in Sections \ref{3OM} and \ref{4MOD}.

Using the Hilbert space identification
$S^2(H_1,H_2) \simeq\overline{H_1}\stackrel{2}{\otimes} H_2$, Lemma
\ref{1Opp} and (\ref{1Normal}), we have
von Neumann algebra identifications
$$
B(H_1)\overline{\otimes} B(H_2)^{op}\simeq B(H_1)\overline{\otimes} B(\overline{H_2})
\simeq B(H_1\stackrel{2}{\otimes}\overline{H_2})\simeq B(S^2(H_1,H_2))^{op}
$$
and hence a von Neumann algebra embedding
$$
\tau^1\colon M_1\overline{\otimes} M_2^{op}\,\hookrightarrow\,
B(S^2(H_1,H_2))^{op}.
$$
Unraveling the above identifications, we see that
\begin{equation}\label{6Tau1}
\bigl[\tau^1(R\otimes S)\bigr](x)\,=\, SxR,\qquad
x\in S^2(H_1,H_2),\, R\in M_1,\,S\in M_2.
\end{equation}
Further  this property determines $\tau^1$. 
Likewise we may consider 
$$
\tau^3\colon M_2^{op}\overline{\otimes} M_3\hookrightarrow\,
B(S^2(H_2,H_3)),
$$
the (necessarily unique) von Neumann algebra embedding satisfying
\begin{equation}\label{6Tau3}
\bigl[\tau^3(S\otimes T)\bigr](y) = TyS,\qquad y\in S^2(H_2,H_3),\, T\in M_3,\,S\in M_2.
\end{equation}
For convenience, for any $a\in M_1\overline{\otimes} M_2^{op}$ and any
$b\in M_2^{op}\overline{\otimes} M_3$,
we write $\tau^1_a$ instead of $\tau^1(a)$ and $\tau^3_b$ instead of $\tau^3(b)$.

The main objective of this section is to prove the following description of
modular completely bounded $S^1$-multipliers.

\begin{theorem}\label{6Factorization} Assume that $M_2$ is injective.
\begin{itemize}
\item [(a)] Let $I$ be an index set and let
$$
A=(a_i)_{i\in I}\in R_{I}^{w}\bigl(M_1\overline{\otimes} M_2^{op}\bigr)
\qquad\hbox{and}\qquad
B=(b_i)_{i\in I}\in C_{I}^{w}\bigl(M_2^{op}\overline{\otimes} M_3\bigr).
$$
For any $x\in S^2(H_1,H_2)$ and any $y\in S^2(H_2,H_3)$,
$$
\sum_{i\in I} \,\bignorm{\tau^3_{b_i}(y)\tau^1_{a_i}(x)}_1\, <\,\infty.
$$
Let $u_{A,B}\colon S^2(H_2,H_3)\times S^2(H_1,H_2)\to S^1(H_1,H_3)$ be the resulting mapping
defined by
$$
u_{A,B}(y,x)\,=\, \sum_{i\in I} \tau^3_{b_i}(y)\tau^1_{a_i}(x),\qquad
x\in S^2(H_1,H_2),\, y\in S^2(H_2,H_3).
$$
Then $\widetilde{u}_{A,B}\in CB_{(M'_3,M'_2,M'_1)}\bigl(\Gamma(H_1,H_2,H_3), S^1(H_1,H_3)\bigr)$ and
\begin{equation}\label{6cbn}
\cbnorm{\widetilde{u}_{A,B}}\,\leq\, \norm{A}_{R_{I}^{w}}\,\norm{B}_{C_{I}^{w}}.
\end{equation}

\medskip
\item [(b)] Conversely, let $u\colon S^2(H_2,H_3)\times
S^2(H_1,H_2)\to S^1(H_1,H_3)$ be a bounded bilinear map
and assume that $\widetilde{u}$ belongs to
$CB_{(M'_3,M'_2,M'_1)}\bigl(\Gamma(H_1,H_2,H_3), S^1(H_1,H_3)\bigr)$. Then there
exist an index set $I$ and two families
$$
A=(a_i)_{i\in I}\in R_{I}^{w}\bigl(M_1\overline{\otimes} M_2^{op}\bigr)
\qquad\hbox{and}\qquad
B=(b_i)_{i\in I}\in C_{I}^{w}\bigl(M_2^{op}\overline{\otimes} M_3\bigr)
$$
such that $u=u_{A,B}$ and $\norm{A}_{R_{I}^{w}}\,\norm{B}_{C_{I}^{w}} = \cbnorm{u}$.
\end{itemize}
\end{theorem}

We will establish two intermediate lemmas before proceeding to the
proof. We recall the mapping  $\tau$ from (\ref{2Sigma1}). In the sequel
we use the notation $1$
for the unit of either $B(H_1)$ or $B(H_3)$. Thus for any
$a\in M_1\overline{\otimes} M_2^{op}$, we may consider
$a\otimes 1\in M_1\overline{\otimes} M_2^{op}\overline{\otimes}  M_3$.
Likewise, for any
$b\in M_2^{op}\overline{\otimes} M_3$,
we may consider $1\otimes b\in M_1\overline{\otimes} M_2^{op}\overline{\otimes}  M_3$. The following is a generalization
of \cite[Lemma 20]{CLS}.

\begin{lemma}\label{6Magic1}
For any $a\in M_1\overline{\otimes} M_2^{op}$, for any $b\in M_2^{op}\overline{\otimes} M_3$,
and for any
$x\in S^2(H_1,H_2)$ and $y\in S^2(H_2,H_3)$, we have
\begin{equation}\label{6Magic11}
\tau_{(a\otimes 1)(1\otimes b)} (y,x)
\,=\,
\tau^3_{b}(y)\tau^1_{a}(x).
\end{equation}
%Further, $\widetilde{\tau}_{(a\otimes 1)(1\otimes b)}$ belongs to
%$CB\bigl(\Gamma(H_1,H_2,H_3), S^1(H_1,H_3)\bigr)$.
\end{lemma}

\begin{proof}
We fix $x\in S^2(H_1,H_2)$, $y\in S^2(H_2,H_3)$, $\eta_1\in H_1$ and $\xi_3\in H_3$.

Let $R\in M_1,  S,S'\in M_2^{op}, T\in M_3$. 
Then $(R\otimes S\otimes 1)(1\otimes S'\otimes T)
= R\otimes S'S\otimes T$. Hence by (\ref{2Sigma2}), (\ref{6Tau1}) and (\ref{6Tau3}), we have
$$
\tau_{(R\otimes S\otimes 1)(1\otimes S'\otimes T)}(y,x) \,=\, TyS'SxR\,=\, \tau^3_{S'\otimes T}(y)
\tau^1_{R\otimes S}(x).
$$
Hence the result holds true when $a$ and $b$ are elementary tensors. By linearity,
this implies (\ref{6Magic11}) in the case when $a$ and $b$ belong
to the algebraic tensor products $M_1\otimes M_2^{op}$ and
$M_2^{op}\otimes M_3$, respectively.

We now use a limit process. Let $a\in M_1\overline{\otimes} M_2^{op}$
and $b\in M_2^{op}\overline{\otimes} M_3$ be arbitrary. Let
$(a_s)_s$ be a net in $M_1\otimes M_2^{op}$ converging to
$a$ in the $w^*$-topology of $M_1\overline{\otimes} M_2^{op}$
and let $(b_t)_t$ be a net in $M_2^{op}\otimes M_3$ converging to
$b$ in the $w^*$-topology of $M_2^{op}\overline{\otimes} M_3$.
For any $s,t$, we have
\begin{equation}\label{6st}
\tau_{(a_s\otimes 1)(1\otimes b_t)} (y,x)
=\tau^3_{b_t}(y)\tau^1_{a_s}(x)
\end{equation}
by the preceding paragraph.

On the one hand, since the
product is separately $w^*$-continuous on von Neumann algebras,
\begin{equation}\label{6asbt}
(a\otimes 1)(1\otimes b) \,=\,w^*\hbox{-}\lim_s\lim_t (a_s\otimes 1)(1\otimes b_t)
\end{equation}
in $M_1\overline{\otimes} M_2^{op}\overline{\otimes}  M_3$.
Since $\tau$ is $w^*$-continuous, this implies that
$$
\bigl\langle \bigl[\tau_{(a\otimes 1)(1\otimes b)} (y,x)\bigr](\eta_1),\xi_3\bigr\rangle\,=\,
\lim_s\lim_t \bigl\langle \bigl[\tau_{(a_s\otimes 1)(1\otimes b_t)} (y,x)\bigr](\eta_1),\xi_3\bigr\rangle.
$$
On the other hand, by the $w^*$-continuity of $\tau^1$ and $\tau^3$,
$\tau^1_{a_s}\to \tau^1_{a}$ in the $w^*$-topology of $B(S^2(H_1,H_2))$
and $\tau^3_{b_t}\to \tau^3_{b}$ in the $w^*$-topology of $B(S^2(H_2,H_3))$.
Consequently, $\tau^1_{a_s}(x)\to \tau^1_{a}(x)$ in the weak topology of $S^2(H_1,H_2)$ whereas
$\tau^3_{b_t}(y)\to \tau^3_{b}(y)$ in the weak topology of $S^2(H_2,H_3)$. This readily implies that
$$
\bigl\langle \bigl[\tau^3_{b}(y)\tau^1_{a}(x)\bigr](\eta_1),\xi_3\bigr\rangle\,=\,
\lim_s\lim_t
\bigl\langle \bigl[\tau^3_{b_t}(y)\tau^1_{a_s}(x)\bigr](\eta_1),\xi_3\bigr\rangle.
$$
Combining these two limit results with (\ref{6st}), we deduce the formula
(\ref{6Magic11}).
\end{proof}

It follows from Lemma \ref{1Injective1} (a) that
we have $w^*$-continuous and completely isometric
identifications
\begin{equation}\label{6a-alpha}
M_1\overline{\otimes} M_2^{op}\,\simeq\, CB(M_{1*},M_2^{op})
\qquad\hbox{and}\qquad
M_2^{op}\overline{\otimes} M_3\,\simeq\, CB(M_{3*},M_2^{op}).
\end{equation}
Likewise,
$M_1\overline{\otimes} M_2^{op}\overline{\otimes}M_3
\,\simeq\,CB((M_1\overline{\otimes}M_3)_* , M_{2}^{op})$ hence by
\cite[Theorem 7.2.4]{ER},
we have a $w^*$-continuous and completely isometric
identification
\begin{equation}\label{6a-alpha-1}
M_1\overline{\otimes} M_2^{op}\overline{\otimes}M_3
\,\simeq\, CB\bigl(M_{1*}\stackrel{\frown}{\otimes} M_{3*}, M_{2}^{op}\bigr).
\end{equation}

\begin{lemma}\label{6Magic2}
Assume that $M_2$ is injective.
Let $a\in M_1\overline{\otimes} M_2^{op}$ and $b\in M_2^{op}\overline{\otimes} M_3$.
Let $\alpha\in CB(M_{1*}, M_2^{op})$ and $\beta\in CB(M_{3*}, M_2^{op})$
be corresponding to $a$ and $b$, respectively, through the identifications (\ref{6a-alpha}).
Let
$$
\sigma_{a,b}\colon M_{1*}\stackrel{\frown}{\otimes} M_{3*}\longrightarrow M_2^{op}
$$
be the completely bounded map
corresponding to $(a\otimes 1)(1\otimes b)$
through the identification (\ref{6a-alpha-1}). Then we have
\begin{equation}\label{6Magic22}
\sigma_{a,b}(v_1\otimes v_3)\,=\,\alpha(v_1)\beta(v_3)
\end{equation}
for any $v_1\in M_{1*}$ and any $v_3\in M_{3*}$.
\end{lemma}

\begin{proof}
We fix $v_1\in M_{1*}$ and $v_3\in M_{3*}$.

Let $R\in M_1$, $S,S'\in M_2^{op}$ 
and $T\in M_3$, and assume first that $a=R\otimes S$ and
$b=S'\otimes T$. Then $\alpha(v_1)= \langle R,v_1\rangle_{M_1,M_{1*}} S$ and $\beta(v_3)= \langle T, v_3
\rangle_{M_3,M_{3*}} S'$. Hence
$$
\alpha(v_1)\beta(v_3) = \langle R,v_1\rangle_{M_1,M_{1*}} \langle T, v_3
\rangle_{M_3,M_{3*}} S'S.
$$
Since $(a\otimes 1)(1\otimes b) =  R\otimes S'S\otimes T$,
$\sigma_{a,b}(v_1\otimes v_3)$ is also equal to
$\langle R, v_1\rangle_{M_1,M_{1*}}
\langle T, v_3
\rangle_{M_3,M_{3*}} S'S$. This proves the result
in this special case. By linearity, we deduce that (\ref{6Magic22})
holds true when $a$ and $b$ belong
to the algebraic tensor products $M_1\otimes M_2^{op}$ and
$M_2^{op}\otimes M_3$.

As in the proof of the preceding lemma, we deduce the
general case by a limit process.
Let $a\in M_1\overline{\otimes} M_2^{op}$
and $b\in M_2^{op}\overline{\otimes} M_3$ be arbitrary. Let
$(a_s)_s$ be a net in $M_1\otimes M_2^{op}$ converging to
$a$ in the $w^*$-topology of $M_1\overline{\otimes} M_2^{op}$
and let $(b_t)_t$ be a net in $M_2^{op}\otimes M_3$ converging to
$b$ in the $w^*$-topology of $M_2^{op}\overline{\otimes} M_3$. Then
for any $s,t$, let $\alpha_s\in CB(M_{1*}, M_2^{op})$ and $\beta_t\in CB(M_{3*}, M_2^{op})$
be corresponding to $a_s$ and $b_t$, respectively.
By the
preceding paragraph,
$$
\sigma_{a_s,b_t}(v_1\otimes v_3)\,=\,\alpha_s(v_1)\beta_t(v_3)
$$
for any $s,t$.

Since the identifications (\ref{6a-alpha}) are $w^*$-continuous,
$\alpha_s(v_1)\to\alpha(v_1)$ and $\beta_t(v_3)\to\beta(v_3)$
in the $w^*$-topology of $M_2^{op}$. Since the
product is separately $w^*$-continuous on von Neumann algebras, this implies
that
$$
\alpha(v_1)\beta(v_3)\,=\, w^*\hbox{-}\lim_s\lim_t \alpha_s(v_1)\beta_t(v_3).
$$
Next since the identification (\ref{6a-alpha-1}) is $w^*$-continuous, it follows from
(\ref{6asbt}) that
$$
\sigma_{a,b}(v_1\otimes v_3)\,=\, w^*\hbox{-}\lim_s\lim_t
\sigma_{a_s,b_t}(v_1\otimes v_3).
$$
The identity (\ref{6Magic22}) follows at once.
\end{proof}

%It will follow from the proof of Theorem \ref{6Factorization} that
%if $M_2$ is injective, then for any
%$a\in M_1\overline{\otimes} M_2^{op}$ and any $b\in M_2^{op}\overline{\otimes} M_3$,
%$(a\otimes 1)(1\otimes b)$ belongs to
%$M_2^{op}\overline{\otimes}\bigl(M_1\stackrel{w^*h}{\otimes} M_3\bigr)$
%and $\widetilde{\tau}_{(a\otimes 1)(1\otimes b)}$ is completely bounded from
%$\Gamma(H_1,H_2,H_3)$ into $S^1(H_1,H_3)$.
%

Note that if $M_2$ is injective, then by Lemma \ref{1Injective1} (b)
the identification (\ref{6a-alpha-1}) restricts to an identification
between $M_2^{op}\overline{\otimes}\bigl(M_1\stackrel{w^*h}{\otimes} M_3\bigr)$
and $CB\bigl(M_{1*}\stackrel{h}{\otimes} M_{3*}, M_{2}^{op}\bigr)$.
Combining with (\ref{3Ident1}), we deduce a $w^*$-continuous and completely isometric
identification
\begin{equation}\label{6a-alpha-2}
CB_{(M'_3,M'_2,M'_1)}\bigl(\Gamma(H_1,H_2,H_3), S^1(H_1,H_3)\bigr)
\,\simeq\, CB\bigl(M_{1*}\stackrel{h}{\otimes} M_{3*}, M_{2}^{op}\bigr).
\end{equation}
This will be used in the proof below.

\begin{proof}[Proof of Theorem \ref{6Factorization}]

\

\smallskip (a): Consider $x\in S^2(H_1,H_2)$ and $y\in S^2(H_2,H_3)$. We have
\begin{align*}
\sum_{i\in I} \bignorm{\tau^3_{b_i}(y)\tau^1_{a_i}(x)}_1\,
& \leq\, \sum_{i\in I} \bignorm{\tau^3_{b_i}(y)}_2 \bignorm{\tau^1_{a_i}(x)}_2\\
&\leq\,\Bigl(\sum_{i\in I} \bignorm{\tau^3_{b_i}(y)}_2^2\Bigr)^{\frac12} \Bigl(\sum_{i\in I}
\bignorm{\tau^1_{a_i}(x)}_2^2\Bigr)^{\frac12},
\end{align*}
by the Cauchy-Schwarz inequality.

Let $J\subset I$ be a finite subset.
Since $\tau^3$ is a $*$-homomorphism, we have
\begin{align*}
\sum_{i\in J} \bignorm{\tau^3_{b_i}(y)}_2^2\,
& =\, \sum_{i\in J} \bigl\langle
{\tau^3_{b_i}}^*\tau^3_{b_i}(y),y\bigr\rangle_{S^2}\\
& =\,\Bigl\langle \tau^3\Bigl(\sum_{i\in J} b_i^*b_i\Bigr) (y),y\Bigr\rangle_{S^2}\\
%& \leq\, \Bignorm{\sum_{i\in J} {\tau^3_{b_i}}^*\tau^3_{b_i}}\,\norm{y}^2_2\\
%& \leq\, \Bignorm{\tau^3\Bigl(\sum_{i\in J} b_i^*b_i\Bigr)}\,\norm{y}^2_2\\
& \leq\, \Bignorm{\sum_{i\in J} b_i^*b_i}\,\norm{y}^2_2\\
&\leq\, \norm{B}_{C_{I}^{w}}^2\,\,\norm{y}^2_2.
\end{align*}
Since $J$ is arbitrary, this implies that
\begin{equation}\label{6Square1}
\sum_{i\in I} \bignorm{\tau^3_{b_i}(y)}_2^2\,
\leq\, \norm{B}_{C_{I}^{w}}^2\,\,\norm{y}^2_2.
\end{equation}
Likewise,
\begin{equation}\label{6Square2}
\sum_{i\in I} \bignorm{\tau^1_{a_i}(x)}_2^2
\,\leq\,
\norm{A}_{R_{I}^{w}}^2\,\,\norm{x}^2_2.
\end{equation}
This implies
$$
\sum_{i\in I} \,\bignorm{\tau^3_{b_i}(y)\tau^1_{a_i}(x)}_1\, \leq \,
\norm{A}_{R_{I}^{w}}
\norm{B}_{C_{I}^{w}}\norm{x}_2\norm{y}_2,
$$
which allows the definition of $u_{A,B}$.

Let $n\geq 1$ be an integer,  let $x_1,\ldots, x_n \in
S^2(H_1,H_2)$ and let $y_1,\ldots,y_n \in S^2(H_2,H_3)$.
In the space $S^1(\ell^2_n(H_1),\ell^2_n(H_3))$, we have
the equality
$$
\bigl[u_{A,B}(y_k,x_l)\bigr]_{1\leq k,l\leq n}
= \sum_{i\in I} \bigl[\tau^3_{b_i}(y_k)\tau^1_{a_i}(x_l)\bigr]_{1\leq k,l\leq n}.
$$
Further for any $i\in I$, we have
$$
\bignorm{\bigl[\tau^3_{b_i}(y_k)\tau^1_{a_i}(x_l)\bigr]_{1\leq k,l\leq n}}_{
S^1(\ell^2_n(H_1),\ell^2_n(H_3))}\,\leq\,
\Bigl(\sum_{k=1}^n\norm{\tau^3_{b_i}(y_k)}_2^2\Bigr)^{\frac12}
\Bigl(\sum_{l=1}^n\norm{\tau^1_{a_i}(x_l)}_2^2\Bigr)^{\frac12}.
$$
Consequently, using Cauchy-Schwarz,
\begin{align*}
\bignorm{\bigl[u_{A,B}(y_k,x_l)\bigr]_{1\leq k,l\leq n}}_{S^1(\ell^2_n(H_1),\ell^2_n(H_3))}
\,& \leq\, \sum_{i\in I} \Bigl(\sum_{k=1}^n\norm{\tau^3_{b_i}(y_k)}_2^2\Bigr)^{\frac12}
\Bigl(\sum_{l=1}^n\norm{\tau^1_{a_i}(x_l)}_2^2\Bigr)^{\frac12}\\
&\leq\,\Bigl(\sum_{i\in I} \sum_{k=1}^n\norm{\tau^3_{b_i}(y_k)}_2^2\Bigr)^{\frac12}
\Bigl(\sum_{i\in I}\sum_{l=1}^n\norm{\tau^1_{a_i}(x_l)}_2^2\Bigr)^{\frac12}.
\end{align*}
It therefore follows from (\ref{6Square1}) and (\ref{6Square2})  that
$$
\bignorm{\bigl[u_{A,B}(y_k,x_l)\bigr]_{1\leq k,l\leq n}}_{S^1(\ell^2_n(H_1),\ell^2_n(H_3))}
\,\leq\,\norm{A}_{R_{I}^{w}}\norm{B}_{C_{I}^{w}}\Bigl(\sum_{k=1}^n\norm{y_k}_2^2\Bigr)^{\frac12}
\Bigl(\sum_{l=1}^n\norm{x_l}_2^2\Bigr)^{\frac12}.
$$
According to Lemma \ref{2CB}, this shows that $\widetilde{u}_{A,B} $ is completely bounded
and that (\ref{6cbn}) holds.

Again let $x\in S^2(H_1,H_2)$ and $y\in S^2(H_2,H_3)$.
Using a simple approximation process,
one can check that for any
$R\in M_1'$, $S\in M_2'$ and $T\in M_3'$, we have
$$
\tau^1_a(xR)=\tau^1_a(x)R,\quad
\tau^1_a(Sx)=S\tau^1_a(x),\quad
\tau^3_b(yS)= \tau^3(y)S\quad\hbox{and}\quad
\tau^3_b(Ty)=T\tau^3(y)
$$
whenever $a\in M_1\overline{\otimes} M_2^{op}$ and $b\in M_2^{op}\overline{\otimes} M_3$.
This implies that $(y,x)\mapsto
\tau^3_b(y)\tau^1_a(x)$ is an $(M_3',M_2',M_1')$-module map for any
$a\in M_1\overline{\otimes} M_2^{op}$ and $b\in M_2^{op}\overline{\otimes} M_3$.
This readily implies that $u_{A,B}$ is an $(M_3',M_2',M_1')$-module map.

\bigskip
(b): Assume that $\widetilde{u}\in CB_{(M'_3,M'_2,M'_1)}
\bigl(\Gamma(H_1,H_2,H_3), S^1(H_1,H_3)\bigr)$. Let
$$
\sigma\colon M_{1*}\stackrel{h}{\otimes} M_{3*}\longrightarrow M_2^{op}
$$
be the completely bounded map corresponding to
$\widetilde{u}$ through the identification (\ref{6a-alpha-2}). Since $M_2$ is injective, we may apply
Theorem \ref{4SS} to $\sigma$. We obtain the existence of an index set $I$
and two families $(\alpha_i)_{i\in I}\in R_{I}^{w}\bigl(CB(M_{1*}, M_2^{op})\bigr)$
and $(\beta_i)_{i\in I}\in C_{I}^{w}\bigl(CB(M_{3*}, M_2^{op})\bigr)$ such that
$$
\sigma(v_1\otimes v_3)\,=\,\sum_{i\in I} \alpha_i(v_1)\beta_i(v_3),
\qquad v_1\in M_{1*},\ v_3\in M_{3*}.
$$
For any $i\in I$, we let $a_i\in M_1\overline{\otimes} M_2^{op}$ and
$b_i\in M^{op}_2\overline{\otimes} M_3$ be corresponding to $\alpha_i$ and
$\beta_i$, respectively,  through the identifications (\ref{6a-alpha}). Then we set
$A=(a_i)_{i\in I}$ and $B=(b_i)_{i\in I}$. By Theorem \ref{4SS}, we may assume that
$\norm{A}_{R_{I}^{w}}\norm{B}_{C_{I}^{w}} = \cbnorm{u}$.

For any finite subset $J\subset I$, we may define
$$
u_J\colon S^2(H_2,H_3)\times S^2(H_1,H_2)\to S_1(H_1,H_3)
\qquad\hbox{and}\qquad
\sigma_J\colon M_{1*}\stackrel{h}{\otimes} M_{3*}\to M_2^{op}
$$
by
$$
u_J(y,x)\,=\, \sum_{i\in J} \tau^3_{b_i}(y)\tau^1_{a_i}(x),
\qquad x\in S^2(H_1,H_2),\ y\in S^2(H_2,H_3),
$$
and
$$
\sigma_J(v_1\otimes v_3)\,=\,\sum_{i\in J} \alpha_i(v_1)\beta_i(v_3),
\qquad v_1\in M_{1*},\ v_3\in M_{3*}.
$$
It follows from Lemmas \ref{6Magic1} and \ref{6Magic2} that for any $i$,
the mapping $(v_1\otimes v_3)\to \alpha_i(v_1)\beta_i(v_3)$ corresponds to
the mapping $y\otimes x\mapsto \tau^3_{b_i}(y)\tau^1_{a_i}(x)$ through the identification
(\ref{6a-alpha-2}). By linearity we deduce that $\sigma_J$ corresponds to
$\widetilde{u}_J$ through (\ref{6a-alpha-2}).

We observe that by the easy (and well-known)
converse to Theorem \ref{4SS}, we have
$$
\cbnorm{\sigma_J}\leq \bignorm{(\alpha_i)_{i\in J}}_{R_{J}^{w}(CB(M_{1*}, M_2^{op}))}
\bignorm{(\beta_i)_{i\in J}}_{C_{J}^{w}(CB(M_{3*}, M_2^{op}))}.
$$
This implies the following uniform boundedness,
\begin{equation}\label{6Uniform}
\forall\, J\subset I\ \hbox{finite},\qquad
\cbnorm{\sigma_J}\,\leq\, \norm{A}_{R_{I}^{w}}\norm{B}_{C_{I}^{w}}.
\end{equation}

In the sequel we consider the set of finite subsets of $I$ as directed by inclusion.
We observe that for any $v_1\in M_{1*}$ and $v_3\in M_{3*}$, $\sigma_J(v_1
\otimes v_3)\to \sigma(v_1\otimes v_3)$ in
the $w^*$-topology of $M_2^{op}$. Using the uniform boundedness (\ref{6Uniform}), this implies
that $\sigma_J\to \sigma$ in the point-$w^*$-topology of $CB\bigl(
M_{1*}\stackrel{h}{\otimes} M_{3*},M_2^{op}\bigr)$. Applying (\ref{6Uniform}) again, we deduce
that $\sigma_J\to \sigma$ in the $w^*$-topology of $CB\bigl(
M_{1*}\stackrel{h}{\otimes} M_{3*},M_2^{op}\bigr)$. Since the identification (\ref{6a-alpha-2}) is
a $w^*$-continuous one, this implies that $\widetilde{u}_J\to \widetilde{u}$
is the $w^*$-topology of $CB\bigl(\Gamma(H_1,H_2,H_3), S^1(H_1,H_3)\bigr)$.

Let $x\in S^2(H_1,H_2)$ and $y\in S^2(H_2,H_3)$. The above implies
that $u_J(y,x)\to u(y,x)$ in the $w^*$-topology
of $S^1(H_1,H_3)$. However by part (a) of the theorem,
$$
u_J(y,x)\,\longrightarrow\, \sum_{i\in I} \tau^3_{b_i}(y)\tau^1_{a_i}(x)
$$
in the norm topology of $S^1(H_1,H_3)$. This shows that $u(y,x)$
is equal to this sum, and proves the result.
\end{proof}

\begin{remark} It is clear from its proof that 
part (a) of Theorem \ref{6Factorization} is true without assuming
that $M_2$ is injective. 

The injectivity assumption in Theorem \ref{4SS}
is necessary, see \cite[Theorem 5.3]{SS}, however 
we do not know if it is necessary in part (b) 
of Theorem \ref{6Factorization}.
\end{remark}

The next corollary follows from the above proof.

\begin{corollary}\label{6Facto-varphi}
Assume that $M_2$ is injective and let $\varphi\in M_1\overline{\otimes} M_2^{op}\overline{\otimes} M_3$.
Then $\tau_\varphi$ is a completely bounded $S^1$-multiplier if and only if there
exist an index set $I$ and families
$$
(a_i)_{i\in I}\in R_{I}^{w}\bigl(M_1\overline{\otimes} M_2^{op}\bigr)
\qquad\hbox{and}\qquad
(b_i)_{i\in I}\in C_{I}^{w}\bigl(M_2^{op}\overline{\otimes} M_3\bigr)
$$
such that
$$
\varphi\,=\, \sum_{i\in I} (a_i\otimes 1)(1\otimes b_i),
$$
where the convergence in taken in the $w^*$-topology.
Further 
$$
\cbnorm{\tau_\varphi}\,=\,\inf\Bigl\{
\bignorm{(a_i)_{i\in I}}_{R_I^\omega}\bignorm{(b_i)_{i\in I}}_{C_I^\omega}\Bigr\},
$$
where the infimumm runs over all possible families 
$(a_i)_{i\in I}$ and $(b_i)_{i\in I}$ providing such a factorization of 
$\varphi$.
\end{corollary}

\begin{remark}\label{6Recover}
\

(a)$\,$
Assume that $H_2=\Cdb$ is trivial. Then
$$
\Gamma(H_1,\Cdb,H_3) = \{H_3\}_c\stackrel{\frown}{\otimes} \{\overline{H_1}\}_r
\simeq S^1(H_1,H_3),
$$
by (\ref{1RC}). Hence $CB\bigl(\Gamma(H_1,\Cdb,H_3),S^1(H_1,H_3)\bigr)\simeq CB(S^1(H_1,H_3))$
and in this identification,
$CB_{(M_3',\Cdb,M_1')}\bigl(\Gamma(H_1,\Cdb,H_3),S^1(H_1,H_3)\bigr)$ coincides
with $CB_{(M_3',M_1')}(S^1(H_1,H_3))$, the space of all $(M_3',M_1')$-bimodule
completely bounded maps from $S^1(H_1,H_3)$ into itself.
Further $\tau^1\colon M_1\hookrightarrow 
B(\overline{H_1})^{op}\simeq B(H_1)$
and $\tau^3\colon M_3\hookrightarrow B(H_3)$ 
coincide with the canonical embeddings.
Hence in this case, Theorem \ref{6Factorization}
reduces to  Theorem \ref{Haag} (see also (\ref{Haag+})).

(b)$\,$ A tensor product reformulation of Corollary \ref{6Facto-varphi}
is that the bilinear mapping $(a,b)\mapsto (a\otimes 1)(1\otimes b)$ extends to a
complete quotient map
$$
(M_1\overline{\otimes} M_2^{op})
\stackrel{w^*h}{\otimes} 
(M_2^{op}\overline{\otimes} M_3)\longrightarrow
M_2^{op}\overline{\otimes}\bigl(
M_1 \stackrel{w^*h}{\otimes} 
M_3\bigr).
$$
\end{remark}

We conclude this paper by considering the special case of Schur multipliers.
Our presentation follows \cite{CLS}.  We let $(\Omega_1,\mu_1)$,  $(\Omega_2,\mu_2)$
and $(\Omega_3,\mu_3)$ be three separable
measure spaces. (The separability assumption is not essential but avoids
technical measurability issues.) Recall the classical fact that
to any $f\in L^2(\Omega_1\times\Omega_2)$, one may associate
an operator $x_f\in S^2(L^2(\Omega_1),L^2(\Omega_2))$
given by
$$
x_f(\eta) =\int_{\Omega_1} f(t_1,\,\cdotp)\eta(t_1)\,d\mu_1(t_1),\qquad \eta\in L^2(\Omega_1),
$$
and the mapping $f\mapsto x_f$ is a unitary which yields a Hilbert space identification
$$
L^2(\Omega_1\times\Omega_2)
\,\simeq\,
S^2\bigl(L^2(\Omega_1),L^2(\Omega_2)\bigr).
$$
Of course the same holds with the pairs $(\Omega_2,\Omega_3)$ and $(\Omega_1,\Omega_3)$.
For any $g\in L^2(\Omega_2\times\Omega_3)$ (resp.
$h\in L^2(\Omega_1\times\Omega_3)$) we let $y_g\in
S^2\bigl(L^2(\Omega_2),L^2(\Omega_3)\bigr)$ (resp. $z_h\in S^2\bigl(L^2(\Omega_1),L^2(\Omega_3)\bigr)$) be the
corresponding Hilbert-Schmidt operator.

To any $\varphi\in L^\infty(\Omega_1\times\Omega_2\times\Omega_3)$, one may associate a
bounded bilinear map
$$
\Lambda_\varphi\colon S^2\bigl(L^2(\Omega_2),L^2(\Omega_3)\bigr)\times
S^2\bigl(L^2(\Omega_1),L^2(\Omega_2)\bigr)
\longrightarrow S^2\bigl(L^2(\Omega_1),L^2(\Omega_3)\bigr)
$$
given for any $f\in L^2(\Omega_1\times\Omega_2)$ and $g\in L^2(\Omega_2\times\Omega_3)$ by
$$
\Lambda_\varphi(y_g,x_f)=z_h
$$
where, for almost every $(t_1,t_3)\in \Omega_1\times\Omega_3$,
$$
h(t_1,t_3)\,=\,\int_{\Omega_2}\varphi(t_1,t_2,t_3)f(t_1,t_2)g(t_2,t_3)\,d\mu_2(t_2)\,.
$$
We refer to \cite[Theorem 3.1]{JTT} or \cite[Subsection 3.2]{CLS} for the proof, and also for the fact that
$$
\norm{\Lambda_\varphi\colon S^2\times S^2\longrightarrow S^2}\, =\, \norm{\varphi}_\infty.
$$
Bilinear maps of this form will be called {\bf bilinear Schur multipliers} in the sequel.
Since
$$
S^2\bigl(L^2(\Omega_1),L^2(\Omega_3)\bigr)\,\subset \,
B\bigl(L^2(\Omega_1),L^2(\Omega_3)\bigr)
$$
contractively, we may regard any bilinear Schur multiplier
as valued in $B\bigl(L^2(\Omega_1),L^2(\Omega_3)\bigr)$. Then it follows from
the proof of \cite[Corollary 10]{CLS} that
\begin{equation}\label{5Norm}
\bignorm{\Lambda_\varphi\colon S^2\times S^2\longrightarrow B
\bigl(L^2(\Omega_1),L^2(\Omega_3)\bigr)}\, =\, \norm{\varphi}_\infty.
\end{equation}

For any $i=1,2,3$, let us regard
\begin{equation}\label{5qi*}
L^\infty(\Omega_i)\subset B(L^2(\Omega_i))
\end{equation}
as a von Neumann algebra in the usual way, that is, any $r\in L^\infty(\Omega_i)$
is identified with the multiplication operator $f\mapsto rf,\, f\in L^2(\Omega_i)$.
In the sequel we use the notions considered so far in the case when
$H_i=L^2(\Omega_i)$ and $M_i= L^\infty(\Omega_i)$. We note that
$$
L^\infty(\Omega_i)'=L^\infty(\Omega_i)
\qquad\hbox{and}\qquad L^\infty(\Omega_i)^{op}=L^\infty(\Omega_i).
$$

Using the classical von Neumann algebra identification
$$
L^\infty(\Omega_1\times\Omega_2\times\Omega_3)= L^\infty(\Omega_1)\overline{\otimes}
L^\infty(\Omega_2)\overline{\otimes}
L^\infty(\Omega_3),
$$
we may
apply the construction from Sections 3 and 4 to any
$\varphi\in L^\infty(\Omega_1\times\Omega_2\times\Omega_3)$ and consider
the operator multiplier
$$
\tau_\varphi
\colon S^2\bigl(L^2(\Omega_2),L^2(\Omega_3)\bigr)\times
S^2\bigl(L^2(\Omega_1),L^2(\Omega_2)\bigr)
\longrightarrow B\bigl(L^2(\Omega_1),L^2(\Omega_3)\bigr).
$$
It turns out that
\begin{equation}\label{5tau=lambda}
\tau_\varphi=\Lambda_\varphi.
\end{equation}
The easy verification is left to the reader.

The next proposition should be compared with \cite[Theorem 3.1]{JTT}. In the
latter result, the authors established a similar characterization
of bilinear module maps, but under the assumption that they take values in
$S^2\bigl(L^2(\Omega_1),L^2(\Omega_3)\bigr)$.

\begin{proposition}\label{6Schur1}
For any
$$
u\in B_2\bigl(S^2\bigl(L^2(\Omega_2),L^2(\Omega_3)\bigr)\times
S^2\bigl(L^2(\Omega_1),L^2(\Omega_2)\bigr),
B\bigl(L^2(\Omega_1),L^2(\Omega_3)\bigr)\bigr),
$$
the following are equivalent.
\begin{itemize}
\item [(i)] $u$ is a bilinear Schur multiplier.
\item [(ii)] $u$ is an $(L^\infty(\Omega_3), L^\infty(\Omega_2),L^\infty(\Omega_1))$-module map.
\end{itemize}
\end{proposition}

\begin{proof}
The implication ``(i)$\,\Rightarrow\,$(ii)" follows from (\ref{5tau=lambda}) and Corollary \ref{3Mod-B}.
(It is also possible to write a direct proof.)

To prove the converse,
assume that $u$ is $(L^\infty(\Omega_3), L^\infty(\Omega_2),L^\infty(\Omega_1))$-modular.
We let
$$
U\colon S^1(L^2(\Omega_1))\times S^1(L^2(\Omega_2))
\times S^1(L^2(\Omega_3))\longrightarrow \Cdb
$$
be the unique trilinear form satisfying
$$
U(\overline{\xi_1}\otimes\eta_1, \overline{\xi_2}\otimes\eta_2,
\overline{\xi_3}\otimes\eta_3)
\,=\, \bigl\langle \bigl[u(\overline{\xi_2}\otimes\eta_3, \overline{\xi_1}\otimes\eta_2)\bigr]
(\eta_1),\xi_3\bigr\rangle
$$
for any $\xi_1,\eta_1\in L^2(\Omega_1)$, $\xi_2,\eta_2\in L^2(\Omega_2)$
and $\xi_3,\eta_3\in L^2(\Omega_3)$.

Then for $i=1,2,3$,
let
$$
q_i\colon S^1(L^2(\Omega_i))\longrightarrow L^1(\Omega_i)
$$
be the unique bounded operator satisfying $q_i(\overline{\xi_i}\otimes\eta_i)
=\overline{\xi_i}\eta_i$ for any $\xi_i,\eta_i\in L^2(\Omega_i)$. This is a quotient map,
whose adjoint coincides with the embedding (\ref{5qi*}).

Recall the operators $U_1,U_2,U_3$ defined at the beginning of Section \ref{4MOD}.
By Lemma \ref{3LemMod}, $U_i$ is valued in $L^\infty(\Omega_i)$ for any $i=1,2,3$. This implies
that $U$ vanishes on the union of ${\rm Ker}(q_1)\times S^1(L^2(\Omega_2))
\times S^1(L^2(\Omega_3))$, $S^1(L^2(\Omega_1))\times {\rm Ker}(q_2)
\times S^1(L^2(\Omega_3))$ and $S^1(L^2(\Omega_1))
\times S^1(L^2(\Omega_2))\times {\rm Ker}(q_3)$. Consequently, there exists
a trilinear form
$$
\widehat{u}\colon L^1(\Omega_1)\times L^1(\Omega_2)\times L^1(\Omega_3)
\longrightarrow\Cdb
$$
factorizing $U$ in the sense that
$$
U(v_1,v_2,v_3) = \widehat{u}\bigl(q_1(v_1), q_2(v_2),q_3(v_3)\bigr),
\qquad v_i\in S^1(L^2(\Omega_i)).
$$
Since $L^1(\Omega_1)\widehat{\otimes} L^1(\Omega_2)\widehat{\otimes}  L^1(\Omega_3) =
L^1(\Omega_1\times \Omega_2\times \Omega_3)$ (see e.g. \cite[Chap. VIII, Example 10]{DU}),
there exists $\varphi\in L^\infty(\Omega_1\times \Omega_2\times \Omega_3)$ such that
$$
\widehat{u}(\phi_1,\phi_2,\phi_3)\,=\,\int_{\Omega_1\times\Omega_2\times\Omega_3}
\varphi(t_1,t_2,t_3) \phi_1(t_1)\phi_2(t_2)
\phi_3(t_3)\,d\mu_1(t_1)d\mu_2(t_2)d\mu_3(t_3)
$$
for any $\phi_i\in L^1(\Omega_i)$. A thorough look at the definitions of
$U$ and $\Lambda_\varphi$ then reveals that
$u=\Lambda_\varphi$.
\end{proof}

Combining (\ref{5Norm}), (\ref{5tau=lambda}) and Proposition \ref{2OM-B}, we obtain
that any bilinear Schur multiplier $u$ induces a completely bounded
$$
\widetilde{u}\colon \Gamma\bigl(L^2(\Omega_1), L^2(\Omega_2),L^2(\Omega_3)\bigr)
\longrightarrow B(L^2(\Omega_1), L^2(\Omega_3))
$$
and that $\cbnorm{\widetilde{u}}=\norm{\widetilde{u}} (=\norm{u})$.

The next result, which essentially follows from \cite{CLS},
shows that similarly, $S^1$-Schur multipliers
are automatically completely bounded and that their norm
and completely bounded norm coincide.

\begin{theorem}\label{6Schur2}
Let $\varphi\in L^\infty(\Omega_1\times\Omega_2\times\Omega_3)$.
\begin{itemize}
\item [(a)] $\Lambda_\varphi$ is an $S^1$-operator multiplier if and only if
there exist a separable Hilbert space $H$ and two functions
$$
a\in L^{\infty}(\Omega_1 \times \Omega_2 ; H) \qquad
\text{and} \qquad
b\in L^{\infty}(\Omega_2\times \Omega_3 ; H)
$$
such that
\begin{equation}\label{6Facto1}
\varphi(t_1,t_2,t_3)= \left\langle a(t_1,t_2), b(t_2,t_3) \right\rangle
\end{equation}
for a.e. $(t_1,t_2,t_3) \in \Omega_1 \times \Omega_2 \times \Omega_3.$

In this case,
$$
\bignorm{\Lambda_\varphi \colon  S^2 \times S^2 \rightarrow S^1}=
\inf\bigl\{\norm{a}_\infty\norm{b}_\infty\bigr\},
$$
where the infimum runs over all pairs $(a,b)$ verifying the above factorization property.

\item [(b)] If $\Lambda_\varphi$ is an $S^1$-operator multiplier, then
$$
\widetilde{\Lambda_\varphi}\colon
\Gamma\bigl(L^2(\Omega_1), L^2(\Omega_2),L^2(\Omega_3)\bigr)
\longrightarrow S^1\bigl(L^2(\Omega_1), L^2(\Omega_3)\bigr)
$$
is completely  bounded, with $\cbnorm{\widetilde{\Lambda_\varphi}} =
\norm{\widetilde{\Lambda_\varphi}}$.
\end{itemize}
\end{theorem}

\begin{proof}
Part (a) is given by  
\cite[Theorem 22]{CLS}.

Assume that $\Lambda_\varphi$ is an $S^1$-operator multiplier.
Let 
$$
\S_{3,1}\subset B\bigl(S^\infty(L^2(\Omega_3), L^2(\Omega_1)),
B(L^2(\Omega_3), L^2(\Omega_1))\bigr)
$$
be the space of all measurable Schur multipliers
from $L^2(\Omega_3)$ into $L^2(\Omega_1)$, in the sense of 
\cite[Subsection 2.4]{CLS}. Then using the
notation from the latter paper
(to which we refer for more  explanations), part (a)
implies that $\varphi\in L^\infty_\sigma\bigl(\Omega_2;\S_{3,1}\bigr)$.
Indeed this follows from Peller's description of
measurable Schur multipliers given by \cite[Theorem 1]{Pe}
(see also \cite[Theorem 3.3]{Sp}, 
\cite[Theorem 23]{CLS} and \cite{H}).
Measurable Schur multipliers are
$(L^\infty(\Omega_1),L^\infty(\Omega_3))$-bimodule maps hence by 
\cite[Theorem 2.1]{Sm}, any element of $\S_{3,1}$ is a completely
bounded map, whose completely bounded norm coincides with its usual norm.
Thus we have
$$
\S_{3,1}\,\subset\, CB
\bigl(S^\infty(L^2(\Omega_3), L^2(\Omega_1)),
B(L^2(\Omega_3), L^2(\Omega_1))\bigr)
\qquad\hbox{isometrically}.
$$
We deduce that 
$$
\varphi\in
L^\infty_\sigma\bigl(\Omega_2;CB
\bigl(S^\infty(L^2(\Omega_3), L^2(\Omega_1)),
B(L^2(\Omega_3), L^2(\Omega_1))\bigr)\bigr).
$$
Recall that by \cite[Theorem 2.2]{BS} (see also Theorem 4.2
in the latter paper), we have a $w^*$-continuous isometric identification
$$
CB
\bigl(S^\infty(L^2(\Omega_3), L^2(\Omega_1)),
B(L^2(\Omega_3), L^2(\Omega_1))\bigr)
\simeq B(L^2(\Omega_1))\stackrel{w^*h}{\otimes}
B(L^2(\Omega_3)).
$$
Hence we obtain that $\varphi$ belongs to  
$L^\infty_\sigma\bigl(\Omega_2;B(L^2(\Omega_1))\stackrel{w^*h}{\otimes}
B(L^2(\Omega_3))\bigr)$. Equivalently,
$\varphi$ belongs to $L^\infty(\Omega_2)\overline{\otimes}
\bigl( B(L^2(\Omega_1))\stackrel{w^*h}{\otimes}
B(L^2(\Omega_3))\bigr)$. Moreover 
the norm of
$\Lambda_\varphi \colon  S^2 \times S^2 \rightarrow S^1$ is equal to
the norm of $\varphi$ in the latter space.

Now applying Theorem \ref{2OM-S1}, we deduce that
$\Lambda_\varphi \colon  S^2 \times S^2 \rightarrow S^1$
is completely bounded, with
$\cbnorm{\widetilde{\Lambda_\varphi}} =
\norm{\widetilde{\Lambda_\varphi}}$.
\end{proof}

\begin{remark}
In Theorem \ref{6Schur2} above, (a) can be deduced from (b) as follows. 
Assume that $\Lambda_\varphi$ is a completely bounded $S^1$-operator multiplier,
with completely bounded norm $<1$.
By Proposition \ref{6Schur1} and (\ref{5tau=lambda}), $\Lambda_\varphi=\tau_\varphi$ 
is $(L^\infty(\Omega_3),L^\infty(\Omega_2),L^\infty(\Omega_1))$-modular. Further
$L^\infty(\Omega_2)$ is injective. Hence by Corollary \ref{6Facto-varphi}, 
there exist
an index set $I$, a family $(a_i)_{i\in I}$ in $L^{\infty}(\Omega_1 \times \Omega_2)$
and a family $(b_i)_{i\in I}$ in $L^{\infty}(\Omega_2 \times \Omega_3)$ such that
$$
\sum_{i\in I} \vert a_i\vert^2\, < 1
\qquad\hbox{and}\qquad
\sum_{i\in I} \vert b_i\vert^2\, < 1
$$
almost everywhere on $\Omega_1 \times \Omega_2$ and on 
$\Omega_2 \times \Omega_3$, respectively, and
$\varphi=\sum_{i\in I} (a_i\otimes 1)(1\otimes b_i)$ in
the $w^*$-topology of $L^\infty(\Omega_1\times\Omega_2\times\Omega_3)$.
Since we assumed that the three measure spaces $(\Omega_j,\mu_j)$ are
separable, it follows from the proof of Corollary \ref{6Facto-varphi}
that $I$ can be chosen to be a countable set.
Then we have
\begin{equation}\label{6Facto2}
\varphi(t_1,t_2,t_3)\, = \,\sum_{i\in I} a_i(t_1,t_2)  b_i(t_2,t_3) 
\end{equation}
for a.e. $(t_1,t_2,t_3) \in \Omega_1 \times \Omega_2 \times \Omega_3$.
Further we may define $a\in L^{\infty}(\Omega_1 \times \Omega_2 ; \ell^2_I)$
and $b\in L^{\infty}(\Omega_2\times \Omega_3 ; \ell^2_I)$
by $a(t_1,t_2) =(a_i(t_1,t_2))_{i\in I}$ and 
$b(t_2,t_3) =(b_i(t_2,t_3))_{i\in I}$, respectively.
Then we both have $\norm{a}_\infty\leq 1$ and $\norm{b}_\infty\leq 1$,
and the identity (\ref{6Facto2}) yields (\ref{6Facto1}), with $H=\ell^2_I$. 

Note however we do not know any direct proof of 
Theorem \ref{6Schur2} (b), not using some of the 
arguments from \cite{CLS}.
\end{remark}

\smallskip
\noindent
{\bf Acknowledgements.}
The first author was supported by the French
``Investissements d'Avenir" program,
project ISITE-BFC (contract ANR-15-IDEX-03).

We warmly thank the referee for the careful reading and
several valuable suggestions which improved the presentation
of the paper.

\bigskip

\end{document}